\documentclass[12pt]{amsart}
\usepackage[usenames,dvipsnames]{color}
\usepackage{amssymb}
\usepackage{graphics}
\usepackage{color}

\input prepictex   \input pictex    \input postpictex

\numberwithin{equation}{section}
\newtheorem{theorem}[equation]{\sc Theorem}
\newtheorem{lemma}[equation]{\sc Lemma}
\newtheorem{corollary}[equation]{\sc Corollary}
\newtheorem{proposition}[equation]{\sc Proposition}
\newtheorem{remark}[equation]{\sc Remark }

\newtheoremstyle{expl}{3pt}{3pt}{}{}{\itshape}{.}{.5em}{%
  {\sc\thmname{#1} \theequation}}
\theoremstyle{expl}
\newtheorem{definition}[equation]{Definition}
\newtheorem{example}[equation]{Example}

\def\mput #1 #2/{\put{$\bullet_{#1}$} [tl] <-1mm,1mm> at #2}
\def\arr#1#2{\arrow <2mm> [0.25,0.75] from #1 to #2}
\def\ssize{\scriptstyle}
\def\sssize{\scriptscriptstyle}
\def\len{\operatorname{\ell}}
\def\init{\operatorname{init}}
\def\per{\operatorname{per}}
\def\mult{\operatorname{mult}}
\def\wf{\operatorname{wf}}
\def\fin{\operatorname{fin}}
\def\ind{\operatorname{ind}}

\def\stairgeq{\geq_{\rm stair}}
\def\stairleq{\leq_{\rm stair}}
\def\wfleq{\leq_{\rm wf}^*}
\def\flt{<_{\mathcal F}}
\def\fleq{\leq_{\mathcal F}}
\def\fgt{>_{\mathcal F}}
\def\fgeq{\geq_{\mathcal F}}
\def\op{{\rm op}}
\def\iff{\Longleftrightarrow}
\definecolor{darkgreen}{rgb}{0,0.5,0}
\definecolor{brown}{rgb}{.5,.3,.3}

\newenvironment{blue}{\color{blue}}{}
\newenvironment{magenta}{\color{magenta}}{}
\renewcommand{\color}[1]{}

\begin{document}

\bigskip\bigskip
\begin{center}{\large\bf The Auslander-Reiten Components\\[1ex] in the Rhombic Picture}
\end{center}

\bigskip
\centerline{Markus Schmidmeier and Helene R.\ Tyler}

\begin{center}
\parbox[t]{5.5cm}{\tiny\begin{center}
              Department of Mathematical Sciences\\ 
              Florida Atlantic University\\
              777 Glades Road\\
              Boca Raton, Florida 33431\end{center}}
\parbox[t]{5.5cm}{\tiny\begin{center}
              Department of Mathematics\\ 
              Manhattan College\\
              4513 Manhattan College Parkway\\
              Riverdale, New York 10471\end{center}}

\medskip \parbox[t]{5.5cm}{\centerline{\tiny\tt markus@math.fau.edu}}
         \parbox[t]{5.5cm}{\centerline{\tiny\tt helene.tyler@manhattan.edu}}

\medskip

\begin{center}\tiny Dedicated to Mark Kleiner on the occasion of his 65th birthday\end{center}

\bigskip \parbox{10cm}{\tiny{\bf Abstract.} For an indecomposable module  $M$  
over a path algebra of a quiver of type  $\widetilde{\mathbb A}_n$, 
the Gabriel-Roiter measure gives rise to four new numerical invariants; 
we call them the multiplicity, and the initial, periodic and final parts.  
We describe how these invariants for  $M$ and for its dual specify the position of 
$M$ in the Auslander-Reiten quiver of the algebra.}

\bigskip \parbox{10cm}{\tiny{\it MSC 2010: 16G70 (primary), 
    05E10, 16D70, 16G20}}

\medskip \parbox{10cm}{\tiny{\it Keywords: Auslander-Reiten quiver, Auslander-Reiten sequence, Gabriel-Roiter measure, rhombic picture, string algebra}}

\end{center}

\bigskip

\section{Introduction}

\bigskip

\subsection{The Gabriel-Roiter Measure}
We revisit the traditional partition of the module category 
of a tame hereditary algebra into the preprojective, regular, and preinjective components 
\cite{ARS}, motivated by recent developments in the Gabriel-Roiter theory. 
In particular, the Gabriel-Roiter measure of a module, $\mu(M)$, 
and the dually defined comeasure, $\mu^*(M)$, 
specify the coordinates of the module in the rhombic picture \cite{Ringel05}. 
We focus primarily on tame hereditary $k$-algebras of type ${\widetilde{\mathbb A}_n}$,
where $k$ is an algebraically closed field,
since their modules and homomorphisms are controlled 
by the combinatorics of strings and bands \cite{ButlerRingel87,Krause91}. 
This allows us to build a $\lq\lq$Greedy Algorithm" for computing the Gabriel-Roiter 
measure and comeasure for a $k{\widetilde{\mathbb A}_n}$-module. 
We show that the Gabriel-Roiter measure of a $k{\widetilde{\mathbb A}_n}$-module $M$ 
decomposes into three distinct parts: 

\[\mu(M)=\init(M)\;\cdot\; \per(M)^{\mult(M)}\;\cdot\; \fin(M),\] 

and we extract a great deal of information from this decomposition. 
To begin, we find that the decomposition identifies 
the type of the
Auslander-Reiten component in which the module resides.

\begin{proposition}\label{proposition-type-AR-component} 
Let $M$ be an indecomposable $k{\widetilde{\mathbb A}_n}$-module. 
	\begin{enumerate}
	\item $M$ is preprojective if and only if $\fin(M)< \per(M)$ and $\fin(DM)> \per(DM)$ 
	\item  $M$ is preinjective if and only if $\fin(M)> \per(M)$ and $\fin(DM)< \per(DM)$
	\item  $M$ is regular if and only if $\fin(M)< \per(M)$ and $\fin(DM)< \per(DM)$
	\end{enumerate}

\end{proposition}

Throughout the paper we freely use the Gabriel-Roiter theory terminology introduced by Ringel in \cite{Ringel05}, and the usual representation theory terminology introduced in \cite{Ringel84} (see also the books \cite{ARS} and \cite{ASS}). In particular, we recall from \cite{Ringel84} that the regular component $Reg(A)$ of the Auslander-Reiten quiver $\Gamma(A)$ of a tame hereditary $k$-algebra $A$ of Euclidean type is a disjoint union of a ${\mathbb P}_1(k)$-family of stable tubes.

\subsection{Module Families in the Rhombic Picture}

Assume that $k$ is an algebraically closed field and $A$ is a tame hereditary algebra of Euclidean type. We recall from \cite{Ringel84} that a regular $A$-module $M$ defines a ray and a coray in its tube. We order the modules on the intersection of the ray and the coray by size, obtaining the family $(M_i)_{i\in{\mathbb N}}$ of $M$. For preprojective and preinjective modules we use the cyclic structure of ${\widetilde{\mathbb A}_n}$ to define corresponding families. We define the rhombic limit of $M$ as 

\[{\displaystyle\overrightarrow{\rho}(M)=\left(\lim_{i\to\infty} \mu(M_i),\lim_{i\to\infty} \mu(DM_i)\right)}\]

We will see in Corollary~\ref{corollary-approach-GR-limit} and Proposition~\ref{rhombic-tubes}
that preprojective and regular families approach the GR-limit from below,
while preinjective families approach their GR-limit from above.
This can be expressed in terms of chain conditions on totally ordered sets of
GR-measures.

\begin{definition}
We call a family $\mathcal G$ of GR-measures {\it noetherian} ({\it artinian})
if $(\mathcal G,<)$ satisfies the ascending (decending) chain condition. 
For $\mathcal M$ a class of modules, we write
$$\mathcal G(\mathcal M) \;=\; \big\{ \mu(M) \big| M\in \mathcal M \big\},
\quad  \mathcal G^*(\mathcal M) \;=\;\big\{\mu^*(M)\big| M\in \mathcal M\big\}$$
for the sets of GR-measures and co-measures of modules in $\mathcal M$.
\end{definition}

A recent result by Dung and Simson \cite[Theorem~3.2]{DS} states that a right artinian ring $R$
is right pure semisimple if and only if the set 
$\mathcal G(\ind R)$ is noetherian.  We obtain in our situation:

\begin{corollary}
Let $\mathcal P$, $\mathcal R$, $\mathcal Q$ be full sets of representatives 
of indecomposable preprojective, regular and preinjective 
modules over a path algebra of type $\widetilde{\mathbb A}_n$.
Then the sets $\mathcal G(\mathcal P)$, $\mathcal G(\mathcal R)$ 
and $\mathcal G^*(\mathcal P)$
are artinian but not noetherian, and the sets $\mathcal G(\mathcal Q)$,
$\mathcal G^*(\mathcal R)$, $\mathcal G^*(\mathcal Q)$ are noetherian but not artinian.  
\qed
\end{corollary}

\medskip
The concept of limits is compatible with Auslander-Reiten sequences:

\begin{theorem}\label{theorem-parallelogram} 
Let $0\to A\to B_1\oplus B_2\to C\to 0$ be an Auslander-Reiten sequence 
such that the middle term consists of two indecomposable summands.  
Then the rhombic limits of $A$, $B_1$, $B_2$, $C$ form a (possibly degenerate) parallelogram 
in the rhombic picture.  Moreover, the nondegenerate sides of the parallelogram 
are parallel to the $\mu$ and $\mu^*$ axes.
\end{theorem}

\subsection{The Tiling Theorem}


Finally we present a result that investigates the phenomenon that
for a given tube $\mathcal T$, the sequence in which the modules
on a ray or on a coray occur can be read off from the arrangement 
of their limit points in the rhombic picture.
Namely, we have seen in Theorem~\ref{theorem-parallelogram} 
that for the modules on any ray in a given tube $\mathcal T$,
the GR-limits in the rhombic picture lie on a line $\ell$ 
which is parallel to the $\mu^*$-axis.
Under certain conditions on $Q$ (we say ``$Q$ has a widest hill''), the precise sequence of
module families on the ray can be read off from the rhombic picture by just repeating the sequence
in which their limits occur on $\ell$. 
By combining this result with the corresponding statement for corays, we obtain that the arrangement
of limit points in the rhombic picture for the families in $\mathcal T$ represents the 
sequences in which those families occur along the rays and corays in $\mathcal T$.
In this sense, the arrangement of limit points in the rhombic picture ``provides a tiling for the tube''.

\smallskip
Let $M$ be a regular module, sufficiently large so that each part of its measure 
\[\mu(M)=\init(M)\cdot\per(M)^{\mult(M)}\cdot\fin(M)\] 
is defined. Then the periodic part, $\per(M)$, which
may be of type $\mathcal L$, $\mathcal R$ or $(h)$, 
distinguishes the tube containing $M$. More precisely,
for a fixed orientation of the quiver $Q$, the sequences $\mathcal L$
and $\mathcal R$ consist of left and right hooks, respectively, 
which determine the orientation on the quiver, 
while $h=n+1$ is the dimension of a quasisimple homogeneous module. 
The integer $\mult(M)$ then gives the position of $M$ within its family. 
The remaining $\lq\lq$waist-free" part of the measure
\[\wf(M)=\init(M)\cdot\fin(M)\]
describes the position of the family within the tube. 
This is an invariant of the family, as the waist-free parts are totally ordered. 
In special cases, this ordering agrees with the ordering of the rays within the tube. In section 6.5 we will give a precise definition of a $\lq\lq$widest valley" and $\lq\lq$widest hill"  for an $\widetilde{\mathbb A}_n$ quiver.
\begin{theorem}
Suppose the quiver $Q$ of type $\widetilde{\mathbb A}_n$ has a unique
widest valley and a unique widest hill. Then for each tube,
the system of limits in the rhombic picture provides a tiling for the tube.
\end{theorem}

We give an example to illustrate its statement.

\begin{example}\label{first-example}
We consider one of the exceptional tubes for the two-sink, two-source 
$\widetilde{\mathbb A}_{4}$ quiver. 
$$
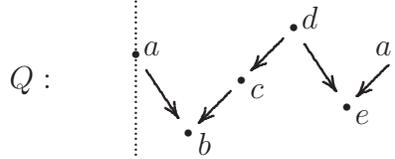
\beginpicture
  \setcoordinatesystem units <0.7cm, 0.7cm>
  \put{$Q:$} at 0 0
  \multiput{$\sssize\bullet$} at 2 .5  3 -1  4 0  5 1  6 -.5  7 .5 /
  \arr{2.2 .2} {2.8 -.7}
  \arr{3.8 -.2} {3.2 -.8} 
  \arr{4.8 .8} {4.2 .2}
  \arr{5.2 .7} {5.8 -.2}
  \arr{6.8 .3} {6.2 -.3}
  \put{$a$} at 2.3 .6
  \put{$b$} at 3.3 -1.2
  \put{$c$} at 4.3 -.2
  \put{$d$} at 5.3 1.2
  \put{$e$} at 6.3 -.7
  \put{$a$} at 6.7 .6
  \setdots<2pt>
  \plot 2 1.5  2 -1.5 /
  \plot 7 1.5  7 -1.5 /
\endpicture
$$
We will see that the quiver has indeed a unique widest valley and a unique
widest hill.

$$
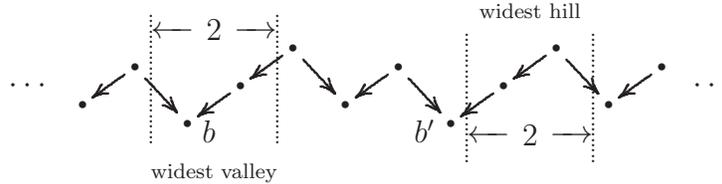
\beginpicture
  \setcoordinatesystem units <.7cm, .5cm>
  \multiput{$\sssize\bullet$} at  6 -.5  7 .5  8 -1  
                       9 0  10 1  11 -.5  12 .5  13 -1
                       14 0  15 1  16 -.5  17 .5 /
  \arr{6.8 .3} {6.2 -.3}
  \arr{7.2 .2} {7.8 -.7}
  \arr{8.8 -.2} {8.2 -.8}
  \arr{9.8 .8} {9.2 .2}
  \arr{10.2 .7} {10.8 -.2}
  \arr{11.8 .3} {11.2 -.3}
  \arr{12.2 .2} {12.8 -.7}
  \arr{13.8 -.2} {13.2 -.8}

  \arr{14.8 .8} {14.2 .2}
  \arr{15.2 .7} {15.8 -.2}
  \arr{16.8 .3} {16.2 -.3}

  \put{$b$} at 8.4 -1.2
  \put{$b'$} at 12.5 -1.2
\setdots<2pt>
\plot 7.3 -1.5  7.3 2 /
\plot 9.7 -1.5  9.7 2 /
\plot 13.3 -2  13.3 1.5 /
\plot 15.7 -2  15.7 1.5 /
\betweenarrows {2} from 7.3 1.5  to 9.7 1.5 
\betweenarrows {2} from 13.3 -1.3  to 15.7 -1.3
\put{\tiny widest valley} at 8.5 -2.3
\put{\tiny widest hill} at 14.5 2 
\multiput{$\cdots$} at 5 0  18 0 /
\endpicture
$$

Here we draw the tube with its mouth at the top, 
to more easily display the correspondence between the arrangement 
of the rhombic limits and of the AR-sequences in the tube.

$$
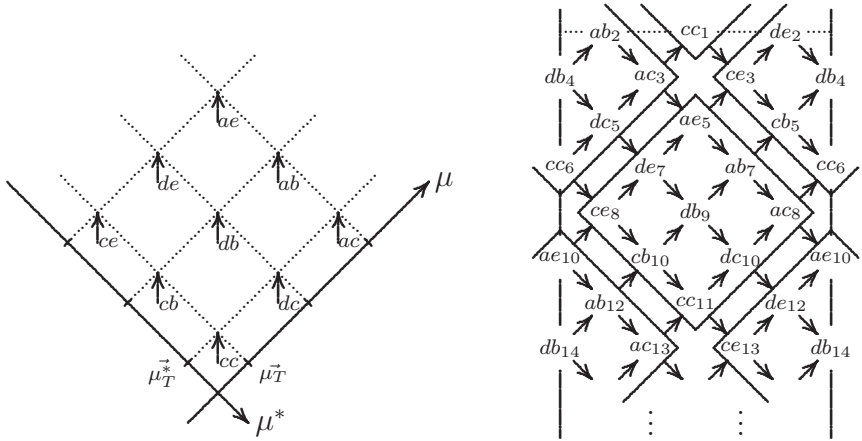
\hbox{\beginpicture
\setcoordinatesystem units <0.2cm,0.2cm>
\put{} at -14 22.5
\put{} at 15 -2
\arr{-14 14}{2 -2}
\arr{-2 -2}{14 14}
\put{$\mu^*$} at 3.5 -2
\put{$\mu$} at 15 14
\setsolid
\plot 1.8 2.2  2.2 1.8 /
\put{$\ssize \vec{\mu_T}$} at 3.7 1.3
\plot -2.2 1.8  -1.8 2.2 /
\put{$\ssize \vec{\mu_T^*}$} at -3.7  1.3
\plot -6.2 5.8  -5.8 6.2 /
\plot -10.2 9.8  -9.8 10.2 /
\plot 6.2 5.8  5.8 6.2 /
\plot 10.2 9.8  9.8 10.2 /
\setdots<2pt>
\plot -2 2  10.5 14.5 /
\plot 2 2    -10.5 14.5 /
\plot -6 6  6.5 18.5 /
\plot 6 6 -6.5 18.5 /
\plot -10 10  2.5 22.5 /
\plot 10 10  -2.5 22.5 /
\setsolid
\begin{magenta}%
  \arr{0 10}{0 12}
  \arr{0 18}{0 20}
  \arr{-4 14}{-4 16}
  \arr{4 14}{4 16}
  \arr{-8 10}{-8 12}
  \arr{-4 6}{-4 8}
  \arr{0 2}{0 4}
  \arr{4 6}{4 8}
  \arr{8 10}{8 12}
  \put{$\ssize ab$} at 4.7 14 
  \put{$\ssize de$} at -3.3 14  
  \put{$\ssize ae$} at .7 18  
  \put{$\ssize db$} at .7 10  
  \put{$\ssize cb$} at  -3.3 6 
  \put{$\ssize ce$} at -7.3 10
  \put{$\ssize dc$} at  4.7 6 
  \put{$\ssize ac$} at 8.7 10
  \put{$\ssize cc$} at .7 2
\end{magenta}%
\endpicture}
\qquad
\hbox{\beginpicture
\setcoordinatesystem units <0.6cm,0.6cm>
\put{} at -2 8.6
\put{} at 6 -1
\put{$\ssize ab_2$} at 0 8
\put{$\ssize cc_1$} at 2 8
\put{$\ssize de_2$} at 4 8
\put{$\ssize db_4$} at -1 7
\put{$\ssize ac_3$} at 1 7
\put{$\ssize ce_3$} at 3 7
\put{$\ssize db_4$} at 5 7
\put{$\ssize dc_5$} at 0 6
\put{$\begin{magenta}\ssize ae_5\end{magenta}$} at 2 6
\put{$\ssize cb_5$} at 4 6
\put{$\ssize cc_6$} at -1 5
\put{$\begin{magenta}\ssize de_7\end{magenta}$} at 1 5
\put{$\begin{magenta}\ssize ab_7\end{magenta}$} at 3 5
\put{$\ssize cc_6$} at 5 5
\put{$\begin{magenta}\ssize ce_8\end{magenta}$} at 0 4
\put{$\begin{magenta}\ssize db_9\end{magenta}$} at 2 4
\put{$\begin{magenta}\ssize ac_8\end{magenta}$} at 4 4
\put{$\ssize ae_{10}$} at -1 3
\put{$\begin{magenta}\ssize cb_{10}\end{magenta}$} at 1 3
\put{$\begin{magenta}\ssize dc_{10}\end{magenta}$} at 3 3
\put{$\ssize ae_{10}$} at 5 3
\put{$\ssize ab_{12}$} at 0 2
\put{$\begin{magenta}\ssize cc_{11}\end{magenta}$} at 2 2
\put{$\ssize de_{12}$} at 4 2
\put{$\ssize db_{14}$} at -1 1
\put{$\ssize ac_{13}$} at 1 1
\put{$\ssize ce_{13}$} at 3 1
\put{$\ssize db_{14}$} at 5 1
\arr{.3 7.7}{.7 7.3}
\arr{2.3 7.7}{2.7 7.3}
\arr{4.3 7.7}{4.7 7.3}
\arr{1.3 7.3}{1.7 7.7}
\arr{3.3 7.3}{3.7 7.7}
\arr{-.7 7.3}{-.3 7.7}
\arr{.3 6.3}{.7 6.7}
\arr{2.3 6.3}{2.7 6.7}
\arr{4.3 6.3}{4.7 6.7}
\arr{1.3 6.7}{1.7 6.3}
\arr{3.3 6.7}{3.7 6.3}
\arr{-.7 6.7}{-.3 6.3}
\arr{.3 5.7}{.7 5.3}
\arr{2.3 5.7}{2.7 5.3}
\arr{4.3 5.7}{4.7 5.3}
\arr{1.3 5.3}{1.7 5.7}
\arr{3.3 5.3}{3.7 5.7}
\arr{-.7 5.3}{-.3 5.7}
\arr{.3 4.3}{.7 4.7}
\arr{2.3 4.3}{2.7 4.7}
\arr{4.3 4.3}{4.7 4.7}
\arr{1.3 4.7}{1.7 4.3}
\arr{3.3 4.7}{3.7 4.3}
\arr{-.7 4.7}{-.3 4.3}
\arr{.3 3.7}{.7 3.3}
\arr{2.3 3.7}{2.7 3.3}
\arr{4.3 3.7}{4.7 3.3}
\arr{1.3 3.3}{1.7 3.7}
\arr{3.3 3.3}{3.7 3.7}
\arr{-.7 3.3}{-.3 3.7}
\arr{.3 2.3}{.7 2.7}
\arr{2.3 2.3}{2.7 2.7}
\arr{4.3 2.3}{4.7 2.7}
\arr{1.3 2.7}{1.7 2.3}
\arr{3.3 2.7}{3.7 2.3}
\arr{-.7 2.7}{-.3 2.3}
\arr{.3 1.7}{.7 1.3}
\arr{2.3 1.7}{2.7 1.3}
\arr{4.3 1.7}{4.7 1.3}
\arr{1.3 1.3}{1.7 1.7}
\arr{3.3 1.3}{3.7 1.7}
\arr{-.7 1.3}{-.3 1.7}
\arr{.3 .3}{.7 .7}
\arr{2.3 .3}{2.7 .7}
\arr{4.3 .3}{4.7 .7}
\arr{1.3 .7}{1.7 .3}
\arr{3.3 .7}{3.7 .3}
\arr{-.7 .7}{-.3 .3}
\setdots<2pt>
\plot -1 8  -.5 8 /
\plot .5 8  1.5 8 /
\plot 2.5 8  3.5 8 /
\plot 4.5 8  5 8 /
\setsolid
\plot -1 8.5  -1 7.5 /
\plot -1 6.5  -1 5.5 /
\plot -1 4.5  -1 3.5 /
\plot -1 2.5  -1 1.5 /
\plot -1 .5   -1 -1 /
\plot 5 8.5  5 7.5 /
\plot 5 6.5  5 5.5 /
\plot 5 4.5  5 3.5 /
\plot 5 2.5  5 1.5 /
\plot 5 .5   5 -1 /
\multiput{$\vdots$} at 1 -.5  3 -.5 /
\begin{magenta}%
  \plot -.6 4  2 6.6  4.6 4  2 1.4  -.6 4 /
  \plot -1.6 5  -1 4.4  1.6 7  0 8.6 /
  \plot .8 8.6  2 7.4  3.2 8.6 /
  \plot 4 8.6  2.4 7  5 4.4  5.6 5 /
  \plot 5.6 3  5 3.6  2.4 1  3.6 -.2 /
  \plot -1.6 3  -1 3.6  1.6 1  .4 -.2 /
\end{magenta}%
\endpicture}
$$

\end{example}

\subsection{Organization of this paper}

In Section~\ref{section-preliminaries} we review and slightly modify the basic 
terminology about the Gabriel-Roiter measure.
In particular, the sequences defining the measure in \cite{Ringel05} are obtained
from the sequences that we will use by taking partial sums.
We also adapt terminology for dealing with string and band modules to the situation where
the underlying quiver is of type $\widetilde{\mathbb A}_n$ 
and  carries a fixed orientation.

\smallskip
Many of the results in our paper will be obtained through a careful analysis of the 
Greedy Algorithm, which we define and discuss in Section~\ref{section-gr-measures}.
With this tool we can decompose the Gabriel-Roiter measure of a module into initial,
periodic and final parts.

\smallskip
The special shape of the quiver of type $\widetilde{\mathbb A}_n$ allows us to define
in Section~\ref{section-families} 
the family $(M_i)$ of a string module $M$ as the modules on the intersection of the ray 
and the coray (or the two rays or corays) given by $M$.  We characterize the components of the 
Auslander-Reiten quiver in terms of properties of the limits of the GR-measures
of the modules $M_i$ and $DM_i$ and obtain a proof for 
Proposition~\ref{proposition-type-AR-component}.

\smallskip
In Section~\ref{section-rhombic} we present a general result. There we consider families of modules in a stable tube for an
arbitrary algebra.  We show that in the rhombic picture, the modules in a family approach
the rhombic limit from below and give the proof for Theorem~\ref{theorem-parallelogram}.

\smallskip
In the final Section~\ref{section-tiling} we return to the case where the quiver has type
$\widetilde{\mathbb A}_n$.   We distinguish the tubes in terms of the periodic parts of the
GR-measures of their modules.  For each tube, we refine the rhombic limit to make 
the parallelogram in Theorem~\ref{theorem-parallelogram} non-degenerate and show that for 
suitable quivers, the system of limit points in the rhombic picture gives rise to a tiling
of the tube.

\section{Preliminaries}\label{section-preliminaries}

\subsection{The Gabriel-Roiter Measure}
The Gabriel-Roiter measure of a module was introduced by Ringel in  \cite{Ringel05}.  
There, one considers all possible chains of indecomposable submodules and imposes an 
ordering on the lists of their lengths.  Here, we begin by making a slight modification to 
Ringel's definition. Namely, instead of recording the lengths of the modules in an indecomposable filtration, 
we list the lengths of the subsequent quotients.  The ordering of the measures remains the same as Ringel's.  
We recall the details.

\smallskip
Let $\mathcal F$ denote the set of all (finite or infinite) sequences $(a_i)$
with entries in $\mathbb N=\{1, 2,\ldots\}$.  Define a total ordering on $\mathcal F$
by putting $(a_i)<_{\mathcal F}(b_i)$ if either $(a_i)$ is a proper subsequence of $(b_i)$ or 
if there exists $\ell \in \mathbb N$ with $a_i=b_i$ for $i<\ell$ and $a_\ell>b_\ell$. 
Note that with this ordering, $\mathcal F$ is a compact and complete metric space.
Consider the map $e:\mathcal F\to \mathbb R$ given by 
$(a_i)\mapsto \sum_\ell 2^{-\sigma_\ell}$ where $\sigma_\ell=\sum_{i\leq\ell}a_i$ is the 
partial sum.  Then $(a_i) <_{\mathcal F}(b_i)$ in $\mathcal F$ if and 
only if $e(a_i)<e(b_i)$ in $\mathbb R$ unless one of the sequences is infinite
and has almost all entries equal to one.

\smallskip
Now let $\Lambda$ be an artin algebra and let $M$ be a $\Lambda$-module.  A sequence $0=M_0\subsetneq M_1\subsetneq \cdots$
of submodules of $M$ where each $M_i$ with $i>0$ is indecomposable
and of finite length is called an {\it indecomposable filtration for} $M$.
The {\it GR-measure} $\mu(M)$ is defined as the supremum taken in $\mathcal F$
of the sequences $(|M_i/M_{i-1}|)$ where
$0=M_0\subsetneq M_1\subsetneq \cdots$ is an indecomposable filtration for $M$.  Whenever 
the supremum is attained,
this chain is called a {\it GR-filtration}.  
As noted in \cite{Ringel05}, if $M$ is finitely generated, 
then a GR-filtration exists if and only if $M$ is indecomposable.

\smallskip
Dually we denote by $\mathcal F^*$ the set $\mathcal F$ with the opposite ordering,
so $(a_i)\flt^* (b_i)$ holds if and only if $(a_i)\fgt (b_i)$. 
We define an {\it indecomposable cofiltration for} $M$
as a sequence $\cdots M_2\subsetneq M_1\subsetneq M_0=M$
where each factor $M/M_i$ for $i\geq 1$ is indecomposable of finite length.
Then the {\it GR-comeasure} $\mu^*(M)$ is the infimum of the sequences
$(|M_{i-1}/M_i|)$ where the $M_i$ form an indecomposable cofiltration for $M$,
and where the infimum is taken in $\mathcal F^*$.
Clearly, if $M$ is a finite dimensional module then $\mu^*(M) = \mu(DM)$.
Here $DM$ is the dual module, which is defined over the opposite algebra.

\subsection{String Modules for $\widetilde{\mathbb A}_n$} For an arbitrary quiver $Q$, the GR-measure of a $kQ$-module may be quite difficult to compute.  But as we will see, when $Q$ is a quiver of type $\widetilde {\mathbb A}_n$, the string algebra structure of $kQ$ (see, for example, \cite{ButlerRingel87} and \cite{Krause91}) allows us to construct an explicit algorithm for computing GR-measures.   For such a quiver $Q$, let $f:\widetilde Q\to Q$ be a universal covering,
so that $\widetilde Q$ is a quiver of type $\mathbb A_\infty^\infty$ with 
vertex set $\mathbb Z$ and the orientation of the 
arrow between $i$ and $i+1$ is given by the arrow between $f(i)$ and $f(i+1)$.
Note that by fixing a covering and labeling the vertices in $\widetilde Q$ with integers, we have 
implicitly chosen an {\it orientation} for $\widetilde{\mathbb A}_n$.  Recall that a string module is determined by its starting point and length. Thus, the covering $f$ induces a one-to-one correspondence between string modules for $kQ$ and intervals in $\widetilde Q$, up to the shift by a fundamental domain.  Moreover, the Auslander-Reiten component to which an indecomposable module belongs is completely determined by its string or band structure.  We call a string preprojective (preinjective, left regular, right regular) if in the corresponding interval in $\widetilde Q$ the two arrows in $\widetilde Q$ neighboring the interval but lying outside of it point towards the interval (away from the interval, to the left, to the right).  Then the preprojective modules, the preinjective modules, and the modules in the two exceptional tubes
are string modules corresponding to preprojective, preinjective, left regular, and right regular strings.  The band modules lie in homogeneous tubes.   Each string is the domain of at most two irreducible morphisms, corresponding to the endpoints of the interval.  At each endpoint we can perform at most one of the following two operations.  If there is an incoming arrow at an endpoint, we can $\lq\lq$ add a hook", by expanding the string to the next incoming arrow.  If there is an outgoing arrow at an endpoint, we $\lq\lq$delete a cohook" by contracting the string to the endpoint of the previous outgoing arrow.  As a consequence, all irreducible maps in the preprojective (preinjective) component are monomorphisms (epimorphisms) since all of the maps are formed by adding hooks (deleting cohooks).

\section{Structure of GR-Measures for $\widetilde{\mathbb A}_n$}\label{section-gr-measures}

\subsection{The Greedy Algorithm for $\widetilde{\mathbb A}_n$ Strings}

\smallskip

\smallskip
Now suppose that $M$ is a string module and that 
$\mathcal M: 0=M_0\subsetneq M_1\subsetneq \cdots \subsetneq M_n=M$ 
is a GR-filtration, such that each $M_i$ is a string module.  This is the case for preprojective modules and for non-homogeneous regular modules.  Then there is a sequence of intervals
$\emptyset=I_0\subsetneq I_1\subsetneq \cdots \subsetneq I_n=I$ in $\widetilde Q$ corresponding 
to the filtration $\mathcal M$.  
Since $\mathcal M$ defines the GR-measure, it follows that $I_1$ consists only 
of a sink in $I$, and each $I_{i+1}$ is obtained from $I_i$ by adding
a left hook or a right hook in $I$. Bearing this correspondence in mind, we now present the
$\lq\lq$Greedy Algorithm", 
which provides an efficient tool for the computation of the GR-measure of $M$.

\smallskip

\begin{enumerate}
\item[(1)] Begin by choosing a sink $\sigma$ in $I$.  Thus, the first entry in the measure is $(1)$.
\item[(2)] Next, compute the sequences $(\lambda_i)$ and $(\rho_i)$ of 
  lengths of left hooks and right hooks, respectively,
  starting at the sink $\sigma$. 
\item[(3)] If $(\lambda_i)<_{\mathcal F}(\rho_i)$, add $\rho_1$ to the measure
  and delete the first entry in the sequence $(\rho_i)$.
  If the inequality is reversed, add $\lambda_1$ to the measure and delete this entry in $(\lambda_i)$.    If $(\lambda_i)=(\rho_i)$, we may make the choice freely.
\item[(4)] Repeat (3) until both sequences $(\lambda_i)$ and $(\rho_i)$ are exhausted.
  This is the highest measure based on the starting point $\sigma$.
\item[(5)] Repeat (1-4) until each sink in $I$ has been used as a starting point. The supremum of the 
resulting sequences is, therefore, the GR-measure of $M$.
\end{enumerate}

\medskip\noindent{\it Example:} Consider the quiver $Q$ of type $\widetilde{\mathbb A}_4$ 
from Example~\ref{first-example},
and the string module $ce_{18}$, as pictured.

$$
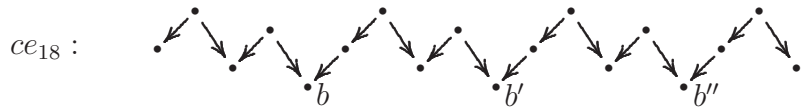
\beginpicture
  \setcoordinatesystem units <0.5cm, 0.5cm>
  \put{$ce_{18}:$} at 1 0
  \multiput{$\sssize\bullet$} at 4 0  5 1  6 -.5  7 .5  8 -1  
                       9 0  10 1  11 -.5  12 .5  13 -1
                       14 0  15 1  16 -.5  17 .5  18 -1
                       19 0  20 1  21 -.5 /
  \arr{4.8 .8} {4.2 .2}
  \arr{5.2 .7} {5.8 -.2}
  \arr{6.8 .3} {6.2 -.3}
  \arr{7.2 .2} {7.8 -.7}
  \arr{8.8 -.2} {8.2 -.8}
  \arr{9.8 .8} {9.2 .2}
  \arr{10.2 .7} {10.8 -.2}
  \arr{11.8 .3} {11.2 -.3}
  \arr{12.2 .2} {12.8 -.7}
  \arr{13.8 -.2} {13.2 -.8}

  \arr{14.8 .8} {14.2 .2}
  \arr{15.2 .7} {15.8 -.2}
  \arr{16.8 .3} {16.2 -.3}
  \arr{17.2 .2} {17.8 -.7}
  \arr{18.8 -.2} {18.2 -.8}
  \arr{19.8 .8} {19.2 .2}
  \arr{20.2 .7} {20.8 -.2}

  \put{$b$} at 8.4 -1.2
  \put{$b'$} at 13.5 -1.2
  \put{$b''$} at 18.6 -1.2 
\endpicture
$$

Note that the module $ce_{18}$ is regular and occurs in an exceptional tube.
Thus any submodule of $ce_{18}$ is either preprojective, or occurs in the same tube.
In each case, the submodule is a string module.  
Hence, the Greedy Algorithm will produce the GR-measure.  Moreover, we need only run the algorithm beginning from sinks that lie in $\lq\lq$widest valleys", as this will maximize the small entries at the beginning of the measure.  Then running the algorithm, we find the following.

\medskip
\begin{tabular}{c|c|c|c}
starting & left hook              & right hook          & candidate for \\ 
point    & sequence $(\lambda_i)$ & sequence $(\rho_i)$ & GR-measure  \\ \hline
$b$ & $(2,2)$ & $(1,2,2,1,2,2,1,2)$ & $(1,1,2,2,1,2,2,1,2,2,2)$ \\
$b'$ & $(2,3,2,2)$ & $(1,2,2,1,2)$ & $(1,1,2,2,1,2,2,3,2,2)$ \\
$b''$ & $(2,3,2,3,2,2)$ & $(1,2)$ & $(1,1,2,2,3,2,3,2,2)$ 
\end{tabular}\\

\smallskip
Comparing the candidates generated by each starting point, we find that
$\mu(ce_{18}) = (1,1,2,2,1,2,2,1,2,2,2)$.\\

\smallskip

Looking back at the example, we notice that the GR-measure
$\mu(ce_{18})$ has a periodic part in the center.
In fact, depending on the starting point of the period, there are 
several choices for a decomposition of the measure into an initial part, a periodic part
and a  terminal part:
\begin{align*}
\mu(ce_{18})=(1,1,2,2,1,2,2,1,2,2,2) & =  (1;122-122-122;2) \\
                                    & =  (11;221-221;222)  \\
                                    & =  (112;212-212;22)
\end{align*}

This partitioning of the GR-measure is typical, as we will see in the next section.
Moreover, we will see that there is a canonical choice for the periodic part of the measure.

\subsection{The Periodic Part.}For a finite sequence of natural numbers $\mathcal K=(k_1,\ldots,k_s)$, 
we denote by $\mathcal K_u=(k_{u+1},k_{u+2},\ldots,k_u)$ the {\it rotation
by} $u$ where the subscripts are taken modulo $s$.

\begin{lemma}
\label{lemma-tails-have-high-measure}
  Let $\mathcal K=(k_1,\ldots,k_s)$ be a finite sequence of natural
  numbers such that
  \begin{enumerate}
  \item $\mathcal K$ is minimal in $\mathcal F$ among its rotations
    $\mathcal K_u$ and
  \item $\mathcal K$ is not a power of a proper subsequence.
  \end{enumerate}
  Then for each $1\leq u\leq s-1$, the 
  tail $(k_{u+1},\ldots,k_s)$ of $\mathcal K$ 
  is strictly larger in $\mathcal F$ than $\mathcal K$.
\end{lemma}

\begin{proof}
Suppose that $(k_{u+1},\ldots,k_s)<_{\mathcal F}\mathcal K$ for some
$1\leq u\leq s-1$.  Together with the minimality of $\mathcal K$ among the
$\mathcal K_i$ we obtain
$$(k_{u+1},\ldots,k_s)<_{\mathcal F} (k_1,\ldots,k_s)
                    <_{\mathcal F} (k_{u+1},\ldots,k_u),$$
so $k_1=k_{u+1},\ldots,k_{s-u}=k_s\qquad(*)$.

\smallskip
We show by induction on $\ell=1,\ldots,s$ that $k_i=k_{s-u+i}$ holds
for each $1\leq i\leq \ell$.
The inequality $\mathcal K<_{\mathcal F}\mathcal K_u$ together with
$(*)$ and the induction assumption yields that $k_{s-u+\ell}\geq k_\ell$.
The inequality $\mathcal K<_{\mathcal F}\mathcal K_{s-u}$
together with the induction assumption implies $k_\ell\geq k_{s-u+\ell}$,
which proves the claim. 

\smallskip
Since $\mathcal K$ is not a power of a proper subsequence, 
we deduce that $u\equiv0\mod s$, finishing the proof of the Lemma.
\end{proof}

We apply this lemma to the sequences 
$\mathcal L=(\ell_i)_{i=1,\ldots,s}$ and $\mathcal R=(r_i)_{i=1,\ldots,t}$ 
of lengths of left hooks and of right hooks in $\widetilde Q$, 
respectively, covering in each case one period
(the quiver $Q$ may consist of several periods).
We assume that $\mathcal L$ and $\mathcal R$ are rotated such that
they are minimal among the $\mathcal L_u$ and $\mathcal R_u$, respectively.  The lemma then reveals a useful, albeit counterintuitive, fact.  If we begin with a minimally rotated hook sequence and $\lq\lq$bite off" a hook from the beginning of the sequence, the remaining partial hook sequence is larger in $\mathcal F$ than the complete one.  We now see what is so greedy about our algorithm.  According to Lemma \ref{lemma-tails-have-high-measure}, once the first hook has been taken from, say, $\mathcal R$, all subsequent hooks from $\mathcal R$ and its full iterations will be taken.

\medskip
Note now that $Q$ is symmetric if and only if $\mathcal L=\mathcal R$
(since each of the sequences determines and is determined by $\widetilde Q$).
If $Q$ is not symmetric, one direction is distinguished.

\begin{definition}  Let $Q$ be an asymmetric quiver and suppose that $\mathcal L$ and $\mathcal R$ are its mimimally rotated left and right hook sequences. We say the
{\bf take-off direction} is to the right (left) if
$\mathcal L<_{\mathcal F}\mathcal R$ ($\mathcal R<_{\mathcal F}\mathcal L$).
\end{definition}

\smallskip
Thus, whenever the Greedy Algorithm can choose between a sequence of 
left hooks starting with $\mathcal L$ and a sequence of right hooks
starting with $\mathcal R$, then it will move in the take-off
direction. In light of this, we introduce the following convention for the reminder of this paper.
From now on we will assume that either $Q$ is symmetric or
that the orientation is chosen such that $\mathcal L <_{\mathcal F}\mathcal R$
holds.  That is, if $Q$ is not symmetric, we assume without loss of generality that {\it its take-off direction is to the right.} In the next section, we will see that the GR-measure has a characteristic decomposition, in which these hook sequences play and important role. Exactly one of the hook sequences may appear strategically in a module's measure, and for a regular module this will distinguish the tube in which it resides. Moreover, the exceptional tubes behave distinctly with respect to the partition of the module category into the take-off, central, and landing parts.

\begin{example}
Returning to the example in the previous section, 
the right hooks have lengths 1, 2, 2, 
and the left hooks have lengths 3, 2.  Hence, the hook length sequences that are
minimal among their rotations 
are $\mathcal R=(2,2,1)$ and  $\mathcal L=(3,2)$. We see that the take-off direction is 
to the right and that the canonical choice for the decomposition is $\mu(ce_{18}) =  (1,1;\,2,2,1,2,2,1;\,2,2,2)$.
\end{example}

\medskip
\subsection{The IPF Decomposition of the GR-Measure.}

We now show that the GR-measure $\mu$ of a module $M$, with dim$M$ sufficiently large, consists of three distinct parts: an {\it initial part}, a {\it periodic part}, 
and possibly a {\it final part}.  That is, 
\[\mu(M)=\init(M)\cdot \per(M)^{\mult(M)}\cdot\fin(M)\] 

\noindent Here we use multiplicative notation to indicate concatenation of sequences. We will see that the periodic part of the GR-measure may be of type 
$\mathcal L$, $\mathcal R$ or $(h)$ where $h=|Q_0|$.  

\smallskip

Suppose that $\mathcal M: 0=M_0\subsetneq M_1\subsetneq \cdots \subsetneq M_n=M$ is a 
GR-filtration of $M$.  We call the module $M_{n-1}$ a {\it GR-submodule} of $M$.  Note that $\mu(M)=\mu(M_{n-1})\cdot$(dim\,$M/M_{n-1})$.  
We first consider modules $M$ such that the GR-submodule of $M$ is a string module.  
This includes all preprojective modules and
all non-homogeneous regular modules.  We have two cases to consider.

\smallskip
Case 1:  The hook length sequences $\mathcal L$ and $\mathcal R$ are not equal.  Recall that we assume then that the take-off direction is to the right.  Given the starting point $b$, a candidate for the GR-measure produced by the Greedy Algorithm has the following form:
\begin{enumerate}
\item First, there is an initial part.  It consists of 
  a leading 1, some of the $\lambda_i$'s before the first occurance
  of $\mathcal L$ to the left of $b$, and all $\rho_j$ 
  before the first occurance of $\mathcal R$ to the right of $b$.
  Note that the length of the initial part is at most $s+t-1$.
\item Next, according to Lemma \ref{lemma-tails-have-high-measure}, there is a periodic part consisting of all full iterations of $\mathcal R$ occurring to the right of $b$.  
\item At this point, there may occur some of the remaining parts of $(\rho_j)$,
  and all the remaining $\lambda_i$ before the start of the first $\mathcal L$.
\item Lemma~\ref{lemma-tails-have-high-measure} now tells us that once the first hook is taken from $\mathcal L$, all subsequent hooks from $\mathcal L$ and its full iterations will be taken.
\item In the end, there may be a final part.
\end{enumerate}

\smallskip

Now we vary the starting point:

\smallskip
If the measure of the part starting at (3) above is less than $\mathcal R$,
we obtain a larger measure by moving the starting point by one period
to the left, thus 
expanding the part in (2) at the expense of (4).  This process can be repeated 
until no copy of $\mathcal L$ remains in the measure.  Conversely, if the measure of the part starting at (3) above is larger than $\mathcal R$, 
we obtain a larger measure by moving the starting point by one period to the right, expanding (4) at the expense of (2).  
As a consequence, we obtain that for a GR-measure,
either the part in (2) or the part in (4) is empty.  Hence the measure $\mu(M)$
consists of an initial part, $\init(M)$, a periodic part, $\per(M)$, which repeats $\mult(M)$ times, and then possibly a final part, $\fin(M)$. That is,

\[\mu(M)=\init(M)\cdot\per(M)^{\mult(M)}\cdot\fin(M)\]


\medskip
Case 2: The hook length sequences $\mathcal L$ and $\mathcal R$ are equal; i.e., the quiver is symmetric.  Then given a starting point $b$, the candidate for measure produced by the Greedy Algorithm has the following form:

\begin{enumerate}
\item There is an inital part, consisting of a leading 1, 
  and all entries from the sequences $(\lambda_i)$ and $(\rho_j)$ 
  leading up to the first occurance of $\mathcal L=\mathcal R$. 
\item The central part is periodic consisting of all sequences $\mathcal R$
  in $(\lambda_i)$ and $(\rho_j)$. 
\item At the end, there may be a final part.
\end{enumerate}

Again, the GR-measure is obtained by choosing a starting point suitably. 
And, again, the measure decomposes as in the previous case.

\bigskip
Now we consider the situation in which the GR-submodule of $M$ is a band module.  
Then $M$ contains in its GR-filtration a quasi-simple homogeneous module $H$
\cite[Corollary~4.5]{Chen08b} or \cite[Theorem B]{Ringel10}.

In \cite{Ringel05} it was shown that if the base field is infinite, 
then $\mu_H$ is not a measure of finite type.  Hence, $\mu_H$ is not a take-off measure.  
It follows that $\mu_H$ exceeds the GR-measure of any preprojective module.

\smallskip
Next we show that a homogeneous module of quasi-length $q$ has measure 
$$\mu(H[q]) \;=\; \mu_H\cdot (h)^{q-1},$$
where $|Q_0|=h=n+1$. For this we use induction on $q$.
The case where $q=1$ has been handled above.  Suppose for the induction step that $\mu(H[q])$
is as  claimed.  Since the only submodules of $H[q+1]$ are either of the form $H[u]$ for some $u\leq q$, or are preprojective and hence have smaller
GR-measure, 
we see that the GR-submodule of $H[q+1]$ is $H[q]$.  The claim follows.  Note that in this situation the terminal part of the measure is empty.\\

It remains for us to compute the measure of a preinjective module. Since all irreducible maps in the preinjective component are epimorphisms, the GR-submodule $X$ of a preinjective module is either preprojective or regular. However, if $X$ is regular, it may be either a string module or a band module. If $X$ is a string module, then the GR-measure of $M$ is computed by the Greedy Algorithm. If $X$ is a band module, then the following lemma is useful.

\begin{lemma}\label{lemma-preinjectives} 
Suppose $M$ is a preinjective module of dimension $qh+r$, where $h=|Q_0|$, $q\geq1$, and $0< r\leq h-1$. Then there is a (homogeneous) band module $N$ of dimension $qh$ that embeds as a submodule of $M$.
\end{lemma}

\begin{proof} We begin by showing that there is a non-zero map $g:H=H[1]\to M$.  We know that for some $s$, there is a non-zero map $f:H[s]\to M$.  If $f$ does not vanish on the quasi-socle $H[1]$, then we may take $g$ to be the restriction of $f$ to $H[1]$.
If, however, $f$ does vanish on $H[1]$, then $f$ factors over a non-zero map ${\tilde f}:H[s-1]\to M$.
Thus, after at most $s-1$ steps we obtain the desired map $g$.
For dimension reasons, $g$ cannot be an epimorphism; hence, $g$ is a monomorphism:
Since the image of $g$ is regular, the kernel of $g$ (which has defect 0) must be regular.
But $H[1]$ being quasi-simple has no non-zero proper regular submodule.

\smallskip
Next, for every $1\leq s\leq q$, we construct by iteration an embedding $H[s] \to M$.
Having already dealt with the base case, we can assume in the induction step that an inclusion $f:H[s]\to M$
is given.  Since $f$ is not a split monomorphism, it factors over $(f',f^{\prime\prime}):H[s+1]\oplus H[s-1]\to M$.
Again, since $f'$ is not an epimorphism, its image is a regular submodule, and hence the kernel of $f'$
is a regular submodule of $H[s+1]$. 
Since the composition $H[1]\to H[s]\to H[s+1] \to M$ is non-zero, $f'$ must be monic.
\end{proof}

As the GR-measures of homogeneous modules increase with the dimension of
the module, as shown above, it follows that $X$ is of the form $H[q]$ where $q$ is given by $\dim M=qh+r$.
In this case, $\mu(M)=\mu_H\cdot (h)^{q-1}\cdot (r)$.

\medskip
In practice, therefore, the GR-measure of a preinjective module is computed as the maximum of the output of the Greedy Algorithm and the measure $\mu_H\cdot(h)^{q-1}\cdot(r)$, since it is usually quite difficult to determine whether the GR-submodule is a string or a band. However, we do know the following:\\

\begin{lemma} Suppose $M$ is an indecomposable module with GR-filtration $\mathcal M:0=M_0\subsetneq M_1\subsetneq\cdots\subsetneq M_n=M$. If there exists some $i\in\{1,2,\dots,n-1\}$ such that $M_i$ is a band module, then the GR-submodule of $M$ is a band module.
\end{lemma}

\begin{proof} Assume $M_i$ is a band module for some $i\in\{1,2,\dots,n-2\}$. 
Then there is a monomorphism $f:M_i\to M_{i+1}$. 
Since an irreducible
map ending at a preinjective module
must be an epimorphism, we have that $M_{i+1}$ is also homogeneous regular.  
Repeat the argument until $i+1=n-1$.
\end{proof}

The following proposition summarizes the results obtained thus far:

\begin{proposition}\label{proposition-parts-of-mu}
Let $Q$ be a quiver of type ${\widetilde {\mathbb A}}_n$ with $|Q_0|=h=n+1$ and let 
$\mathcal L$ and $\mathcal R$ 
denote the minimally rotated left and right hook sequences for $Q$. 
The GR-measure $\mu$ of a large enough 
indecomposable $kQ$-module $M$ has the form
$$\mu(M)\;=\;\init(M)\cdot\per(M)^{\mult(M)}\cdot\fin(M).$$
The initial part, $\init(M)$, is of length at most $s+t-1$,
ending just before the first occurance of a full sequence $\mathcal L$,
$\mathcal R$ or $(h)$, which then is the periodic part.
It occurs with multiplicity $\mult(M)\in\mathbb N$.
Then, the (possibly empty) final part, $\fin(M)$, 
is of length at most $s+t-2$. \qed
\end{proposition}

\begin{corollary}
The comeasure of a large enough indecomposable module $M$ has the form
$$\mu^*(M)=\init(DM)\cdot\per(DM)^{\mult(DM)}\cdot\fin(DM)$$
\end{corollary}

\begin{proof}By definition, $\mu^*(M)=\mu(DM)$.\end{proof}

\begin{remark}
The definition of $\init(M)$, $\per(M)$, $\fin(M)$ and their duals
will be extended in Remark~\ref{remark-small-modules} to modules 
that are not large enough to have a non-trivial periodic part.
\end{remark}

In light of the corollary, we may define 
$\init^*(M)=\init(DM),\ \per^*(M)=\per(DM),\ \mult^*(M)=\mult(DM),\ \fin^*(M)=\fin(DM)$. 
Thus we have introduced eight combinatorial invariants for a module $M$. 
The interactions among them reveal representation-thoretic information about the module. 
For example, we can determine the Auslander-Reiten component in which the module resides. 
Such issues are considered in the sections that follow.

\section{Module Families in the Rhombic Picture} \label{section-families}

\subsection{Rays and Corays}

Over the path algebra $KQ$, each indecomposable module $M$ is the starting point and end point
of at most two irreducible morphisms; hence it lies on two rays or corays.  The intersection of
the two rays, the two corays, or the ray and the coray defines the {\it family} of $M$
as follows.

\smallskip
For this note that as the quiver $Q$ is cyclic, the irreducible morphisms in the 
preprojective component and in the preinjective component can be partitioned into two classes,
clockwise and counterclockwise.  Thus, an indecomposable preprojective module lies on 
two rays given by irreducible monomorphisms, one going in the clockwise direction,
the other in the counterclockwise direction.
Each ray begins with a projective module.  Dually,
each indecomposable preinjective module defines two corays given by irreducible
epimorphisms, one clockwise, the other 
counterclockwise, each ending at an injective module.

Each of the regular modules defines a ray of monomorphisms starting at a quasi-simple module,
and a coray of epimorphims, ending at a quasi-simple module.

\begin{definition}
Let $M$ be an indecomposable module.  The {\bf family} of $M$ consists of the modules
on the intersection of the two rays (if $M$ is preprojective), the two corays (if $M$ is 
preinjective) or the ray and the coray (if $M$ is regular).
We order the modules in the family by dimension.
\end{definition}

\begin{remark} It is easy to see that two string modules lie in the same family if and only if they have the same string type, in the sense that the strings start at the same
vertex and end at the same vertex.
\end{remark}

\begin{lemma}\label{lemma-family-data}
If $M$ and $M^{\prime}$ are large enough and belong to the same family, then $\init(M)=\init(M^{\prime})$, $\per(M)=\per(M^{\prime})$, and $\fin(M)=\fin(M^{\prime})$. That is, $\mu(M)$ and $\mu(M^{\prime})$ only may differ in the multiplicity of the periodic part.
\end{lemma}

\begin{proof}
We have seen that for a preprojective or regular string module, 
the Greedy Algorithm chooses its starting point near one of the ends of the string. 
Hence, for two modules in the same family, 
the GR-measures will differ only in the length of the periodic part. 
We also have seen that for a preinjective module $M$
the GR-measure is the maximum of the output of the Greedy Algorithm 
and $\mu_H\cdot (h)^{q-1}\cdot (r)$, where $\dim M=qh+r$. 
In either case,  the maximum in $\mathcal F$ depends only 
on the early entries of the sequence. 
Hence, either both modules have a GR-submodule that is a string module, 
or both modules have a GR-submodule that is homogeneous. The result follows.
\end{proof} 

\begin{remark}\label{remark-small-modules}
As a consequence of Lemma~\ref{lemma-family-data}, 
the definition of the initial, periodic and final parts of the GR-measure 
can now be extended to indecomposable modules of arbitrary (small) dimension.
Put $$\init(M)=\init(M_i),\quad \per(M)=\per(M_i),\quad \fin(M)=\fin(M_i)$$
if $M_i$ occurs in the family of $M$ and is large enough.
\end{remark}

\subsection{Limits and Colimits}
In \cite{Ringel05}, Ringel introduced the rhombic picture 
$\mathcal F\times\mathcal F^*$, 
which provides a visual organization of the module category.  
There, each module $M$ is given the coordinates $(\mu(M),\mu^*(M))$.  
We apply this notion to module families.

\begin{lemma}
Let $M$ be an indecomposable module 
and let $(M_i)_{i=1}^{\infty}$ denote the family of $M$.  
Then ${\displaystyle\lim_{i\to\infty}\mu(M_i)}$ and ${\displaystyle\lim_{i\to\infty}\mu^*(M_i)}$ both exist.
\end{lemma}

\begin{proof} Note that GR-measures increase along a ray and, likewise, GR-comeasures decrease along a coray.  The existence of the limits then follows from the fact that $\mathcal F$ is a compact and complete metric space.
\end{proof}

\begin{definition} Let $M$ be an indecomposable module in a stable tube and let $(M_i)_{i=1}^{\infty}$ denote the family of $M$.
\begin{enumerate}
\item The {\bf GR-limit} of $M$ is ${\displaystyle\overrightarrow{\mu}(M)=\lim_{i\to\infty} \mu(M_i)}$
\item The {\bf GR-colimit} of $M$ is ${\displaystyle\overrightarrow{\mu^*}(M)=\lim_{i\to\infty} \mu^*(M_i)}$
\item The {\bf rhombic limit} of $M$ is the point $\overrightarrow{\rho}(M)=(\overrightarrow{\mu}(M),\overrightarrow{\mu^*}(M))$ in the rhombic picture.
\end{enumerate}
\end{definition}

\begin{lemma}\label{lemma-preinj}
For $M$ a preinjective module, 
$\fin(M)>_{\mathcal F}\per(M)$ holds.
\end{lemma}

\begin{proof}
The statement is clear if $\per(M)=(h)$.  If $\per(M)=\mathcal L$ (or $\mathcal R$),
note that $M$ is the epimorphic image of an irreducible map given by deleting a cohook
on the left (right) end of the string defining $M$.  Hence, in the Greedy Algorithm,
after the periodic part, the left (right) hook sequence is a non-empty but incomplete
sequence. It is obtained from $\mathcal L$ ($\mathcal R$) by reducing one of the parts
and omitting all following parts.  This partial sequence is the initial part of $\fin(M)$.
\end{proof}

It is clear that GR-limits given by a ray are approached from below, and 
GR-colimits given by a coray are approached from above.  From the lemma we deduce:

\begin{corollary}\label{corollary-approach-GR-limit}
Preinjective families approach the GR-limit from above, and preprojective families 
approach the GR-colimit from below. \qed
\end{corollary}

The following definition is found in \cite{Ringel05}. We state it here for the convenience of the reader.

\begin{definition}
\begin{enumerate}
\item The {\it take-off limit\/} $\vec\mu_T$ is the supremum of the measures in the take-off part of the module category.
\item The {\it landing limit\/} $\vec\mu_L$ the infimum of the measures in the landing part of the module category.
\end{enumerate}
\end{definition}

Following suit, we have the following.

\begin{definition}
By the {\it homogeneous limit\/} $\vec\mu_H$ we denote the limit
of the measures of the modules in a homogeneous tube. 
\end{definition}

In conclusion, we determine from which direction the limit points in the rhombic 
picture are approached.

\begin{proposition}
For an indecomposable module $M$ with family $(M_i)$, the following statements
are equivalent.
\begin{enumerate}
\item $M$ is preprojective.
\item $\per(M)>_{\mathcal F}\fin(M)$ and $\per^*(M)>_{\mathcal F}^*\fin^*(M)$.
\item The points $\rho(M_i)$ approach the rhombic limit 
  $\vec\rho(M)$ from the left.
\end{enumerate}
Moreover, if $M$ is preprojective then the rhombic limit
$\vec\rho(M)=(\vec\mu,\vec\mu^*)$ satisfies 
$$\vec\mu=\vec\mu_T\quad\text{and}\quad\vec\mu_L^*\fleq^*\vec\mu^*\fleq^*\vec\mu_H^*.$$
\end{proposition}

\begin{proof}
It is clear that (2) and (3) are equivalent.  We show that (1) implies (2),
the converse follows from the corresponding results for preinjective and
for regular modules below.  Suppose that $M$
is preprojective.  Then $DM$ is preinjective, so $\fin(DM)>\per(DM)$
by Lemma~\ref{lemma-preinj}.  As the ordering for comeasures is the opposite
of the ordering for measures, $\fin^*(M)<\per^*(M)$ follows.

\smallskip
Regarding the last assertion,
each preprojective module is in the take-off part \cite[Theorem~3.3]{Chen08}, so
the GR-limit is the take-off limit $\vec\mu_T$,
while for each preinjective module, the GR-limit is at most the landing limit $\vec\mu_L$.
The last inequality follows from Lemma~\ref{lemma-preinjectives} as the GR-limit for
a preinjective module is at least the homogeneous limit $\vec\mu_H$.
\end{proof}

Similarly we have:

\begin{proposition}
For an indecomposable module $M$ with family $(M_i)$, the following statements
are equivalent.
\begin{enumerate}
\item $M$ is preinjective.
\item $\per(M)<_{\mathcal F}\fin(M)$ and $\per^*(M)<_{\mathcal F}\fin^*(M)$.
\item In the rhombic picture, 
  the points $\rho(M_i)$ approach $\vec\rho(M)$ from the right.
\end{enumerate}
Moreover, if $M$ is preinjective then the rhombic limit
$\vec\rho(M)=(\vec\mu,\vec\mu^*)$ satisfies $\vec\mu^*=\vec\mu_T^*$
and $\vec\mu_H\fleq\vec\mu\fleq\vec\mu_L$.
\qed
\end{proposition}

For regular modules we obtain:

\begin{proposition}\label{rhombic-tubes}
For an indecomposable module $M$ with family $(M_i)$, the following statements
are equivalent.
\begin{enumerate}
\item $M$ occurs in a tube.
\item $\per(M)>_{\mathcal F}\fin(M)$ and $\per^*(M)<_{\mathcal F}\fin^*(M)$.
\item In the rhombic picture, the points $\rho(M_i)$ approach $\vec\rho(M)$ from below.
\end{enumerate}
\qed
\end{proposition}

\begin{example}\label{first-example-two}
We apply the above results to the two sink-two source quiver from Example~\ref{first-example}.
There are five projective modules, but six
families of preprojective modules, corresponding to the string types
$bb_*$, $bc_*$, $be_*$, $eb_*$, $ec_*$, $ee_*$.  The families approach the rhombic limit $\vec\rho$
from the left; the rhombic limit has the form $\vec\rho=(\vec\mu_T,\vec\mu^*)$
where $\vec\mu_L^*\fleq^*\vec\mu^*\fleq^*\vec\mu_H^*$.

\smallskip
Dually, the preinjective modules have string types $aa_*$, $ad_*$ ,$cd_*$, $dd_*$, $da_*$, $ca_*$. They approach the rhombic limit from the right; the limit has the form $(\vec\mu,\vec\mu_T^*)$
where $\vec\mu_H\fleq \vec\mu \fleq \vec\mu_L$.

\smallskip
The remaining families consist of regular modules, they all approach the rhombic limit from 
below.

\smallskip
All families and their limits are shown in in Figure~\ref{figure-rhombic-one}.

\begin{figure}[ht]
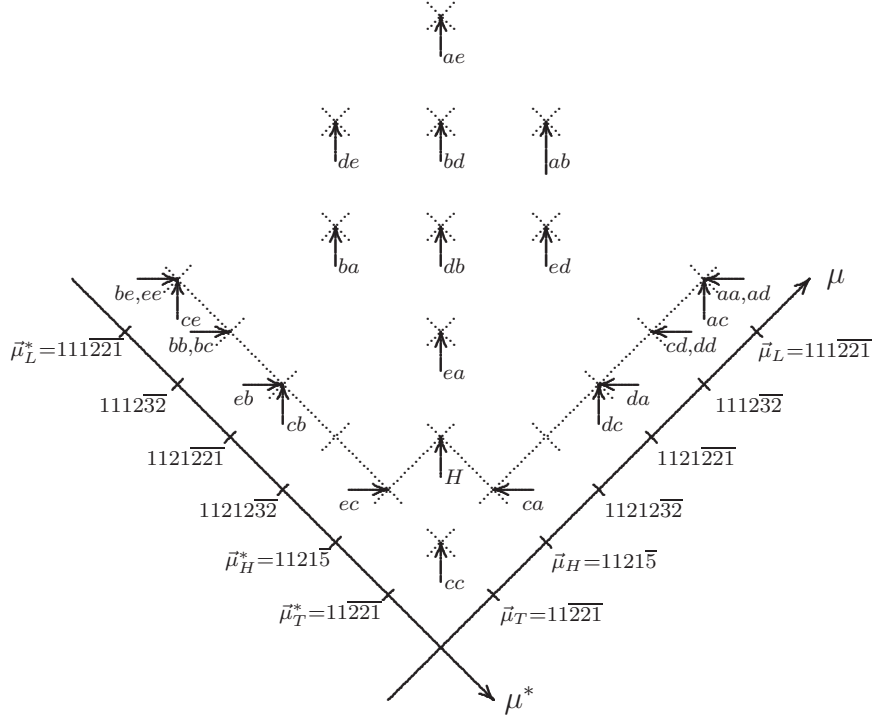

\caption{\label{figure-rhombic-one}The rhombic picture for the first example}
$$
\hbox{\beginpicture
\setcoordinatesystem units <0.35cm,0.35cm>
\def\sput#1{\put{\makebox(0,0){$\ssize #1\strut$}}}
\put{} at -14 24.5
\put{} at 15 -2
\arr{-14 14}{2 -2}
\arr{-2 -2}{14 14}
\put{$\mu^*$} at 3 -2
\put{$\mu$} at 15 14
\setsolid
\plot 1.8 2.2  2.2 1.8 /
\plot 3.8 4.2  4.2 3.8 /
\plot 5.8 6.2  6.2 5.8 /
\plot 7.8 8.2  8.2 7.8 /
\plot 9.8 10.2  10.2 9.8 /
\plot 11.8 12.2  12.2 11.8 /
\put{$\ssize \vec\mu_T=11\overline{221}$} at 4.2 1.3
\put{$\ssize \vec\mu_H=1121\overline 5$} at 6.2 3.3
\put{$\ssize 11212\overline{32}$} at 7.7 5.3
\put{$\ssize 1121\overline{221}$} at 9.7 7.3
\put{$\ssize 1112\overline{32}$} at 11.7 9.3
\put{$\ssize \vec\mu_L=111\overline{221}$} at 14.2 11.3
\plot -2.2 1.8  -1.8 2.2 /
\plot -4.2 3.8  -3.8 4.2 /
\plot -6.2 5.8  -5.8 6.2 /
\plot -8.2 7.8  -7.8 8.2 /
\plot -10.2 9.8  -9.8 10.2 /
\plot -12.2 11.8  -11.8 12.2 /
\put{$\ssize \vec\mu_T^*=11\overline{221}$} at -4.2 1.3
\put{$\ssize \vec\mu_H^*=1121\overline 5$} at -6.2 3.3
\put{$\ssize 11212\overline{32}$} at  -7.7 5.3
\put{$\ssize 1121\overline{221}$} at -9.7 7.3
\put{$\ssize 1112\overline{32}$} at -11.7 9.3
\put{$\ssize \vec\mu_L^*=111\overline{221}$} at -14.2 11.3
\setdots<2pt>
\plot -.5 3.5  .5 4.5 / \plot 1.5 5.5  10.5 14.5 /
\plot -2.5 5.5   .5 8.5 /
\plot -4.5 7.5  -3.5 8.5 /  \plot -.5 11.5  .5 12.5 /  \plot 3.5 15.5  4.5 16.5 /
\plot -6.5 9.5  -5.5 10.5 / \plot -.5 15.5  .5 16.5 /  \plot 3.5 19.5 4.5 20.5 /
\plot -8.5 11.5  -7.5 12.5 /  \plot -4.5 15.5 -3.5 16.5 /  \plot -.5 19.5  .5 20.5 /
\plot -10.5 13.5  -9.5 14.5 /  \plot -4.5 19.5  -3.5 20.5 /  \plot -.5 23.5  .5 24.5 /
\plot .5 3.5  -.5 4.5 / \plot -1.5 5.5  -10.5 14.5 /
\plot 2.5 5.5   -.5 8.5 /
\plot 4.5 7.5  3.5 8.5 /  \plot .5 11.5  -.5 12.5 /  \plot -3.5 15.5  -4.5 16.5 /
\plot 6.5 9.5  5.5 10.5 / \plot .5 15.5  -.5 16.5 /  \plot -3.5 19.5  -4.5 20.5 /
\plot 8.5 11.5  7.5 12.5 /  \plot 4.5 15.5 3.5 16.5 /  \plot .5 19.5  -.5 20.5 /
\plot 10.5 13.5  9.5 14.5 /  \plot 4.5 19.5  3.5 20.5 /  \plot .5 23.5  -.5 24.5 /
\setsolid
{\color{magenta}%
  \arr{0 2.5}{0 4}
  \sput{cc} at .5 2.5
  \arr{0 14.5}{0 16}
  \sput{db} at .5 14.5
  \arr{0 22.5}{0 24}
  \sput{ae} at .5 22.5
  \arr{-6 8.5}{-6 10}
  \sput{cb} at -5.5 8.5
  \arr{6 8.5}{6 10}
  \sput{dc} at 6.5 8.5
  \arr{-10 12.5}{-10 14}
  \sput{ce} at -9.5 12.5
  \arr{10 12.5}{10 14}
  \sput{ac} at 10.5 12.5
  \arr{-4 18.5}{-4 20}
  \sput{de} at -3.5 18.5
  \arr{4 18}{4 20}
  \sput{ab} at 4.5 18.5
}%
{\color{darkgreen}%
  \arr{0 10.5}{0 12}
  \sput{ea} at .5 10.5
  \arr{-4 14.5}{-4 16}
  \sput{ba} at -3.5 14.5 
  \arr{4 14.5}{4 16}
  \sput{ed} at 4.5 14.5
  \arr{0 18.5}{0 20}
  \sput{bd} at .5 18.5
}%
{\color{blue}%
  \arr{-9.5 12}{-8 12}
  \sput{bb,bc} at -9.5 11.5
  \arr{-11.5 14}{-10 14}
  \sput{be,ee} at -11.5 13.5
  \arr{-7.5 10}{-6 10}
  \sput{eb} at -7.5 9.5
  \arr{-3.5 6}{-2 6}
  \sput{ec} at -3.5 5.5
}%
{\color{brown}%
  \arr{9.5 12}{8 12}
  \sput{cd,dd} at 9.5 11.5
  \arr{11.5 14}{10 14}
  \sput{aa,ad} at 11.5 13.5
  \arr{7.5 10}{6 10}
  \sput{da} at 7.5 9.5
  \arr{3.5 6}{2 6}
  \sput{ca} at 3.5 5.5
}%
{\color{red}%
  \arr{0 6.5}{0 8}
  \sput{H} at .5 6.5
}
\endpicture}
$$
\end{figure}
\end{example}

\section{Auslander-Reiten Sequences in the Rhombic Picture}\label{section-rhombic}

Here we present a result that further reveals 
the connection between
Auslander-Reiten sequences and the rhombic picture.
In this section, we put no condition on the underlying algebra.
The following statement holds, as in Proposition \ref{rhombic-tubes}.

\medskip
\begin{remark}
If $M$ is an indecomposable module in
a stable tube, and $(M_i)$ the family of $M$, then the points $\rho(M_i)$
approach the rhombic limit $\vec\rho(M)$ from below.
\end{remark}

\begin{theorem}\label{theorem-parallelograms} Let $0\to A\to B_1\oplus B_2\to C\to 0$ 
be an Auslander-Reiten sequence in a stable tube  
such that the middle term consists of two indecomposable summands.  
Then the rhombic limits of $A,B_1,B_2,C$ form a (possibly degenerate) 
parallelogram in the rhombic picture.  
Moreover, the nondegenerate sides of the parallelogram 
are parallel to the $\mu$ and $\mu^*$ axes.
\end{theorem}

\begin{proof} We may assume that the map $A\to B_1$ is monic; 
otherwise, exchange $B_1$ and $B_2$. 
Now since the modules $A,\, B_1$ lie on the same ray in the tube, 
they have the same GR-limit.   
Hence, their rhombic limits lie on a line parallel to the $\mu^*$ axis 
in the rhombic picture.  
The same is true for the modules $B_2,\, C$.  
Similarly, $A,\, B_2$ are on the same coray in the tube and, 
hence, they have the same GR-colimit.  
Thus, their rhombic limits lie on a line parallel 
to the $\mu$ axis in the rhombic picture.  
The same is true for the modules $B_1,\, C$.  The result follows.
\end{proof}

\begin{example}
For the path algebra of the quiver studied in Example~\ref{first-example} and \ref{first-example-two}
there are two extended tubes.  
The right tube is pictured in Example~\ref{first-example}.
There are nine families, corresponding to the types 
$ab_*,ac_*,ae_*,cb_*,cc_*,ce_*,db_*,dc_*,de_*$.  
Note that all modules in this tube satisfy  $\per=221$ and $\per^*=221$. 

\smallskip
The left tube pictured in Figure~\ref{figure-left} consists of 
four families corresponding to the string types $ba_*,bd_*,ea_*,ed_*$; all have
periodic parts $\per=32$ and $\per^*=32$.  

\smallskip
We see that each Auslander-Reiten sequence in one of the tubes gives rise to 
a non-degenerate parallelogram in the rhombic picture in Figure~\ref{figure-rhombic-one}.
\end{example}

\begin{figure}[ht]
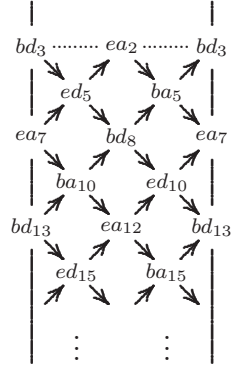

\caption{\label{figure-left}The left tube in the first example}
$$
\hbox{\beginpicture
\setcoordinatesystem units <0.6cm,0.6cm>
\put{} at 0 6.6
\put{} at 4 -1
\put{$\ssize bd_3$} at 0 6
\put{$\ssize ea_2$} at 2 6
\put{$\ssize bd_3$} at 4 6
\put{$\ssize ed_5$} at 1 5
\put{$\ssize ba_5$} at 3 5
\put{$\ssize ea_7$} at 0 4
\put{$\ssize bd_8$} at 2 4
\put{$\ssize ea_7$} at 4 4
\put{$\ssize ba_{10}$} at 1 3
\put{$\ssize ed_{10}$} at 3 3
\put{$\ssize bd_{13}$} at 0 2
\put{$\ssize ea_{12}$} at 2 2
\put{$\ssize bd_{13}$} at 4 2
\put{$\ssize ed_{15}$} at 1 1
\put{$\ssize ba_{15}$} at 3 1
\arr{.3 5.7}{.7 5.3}
\arr{2.3 5.7}{2.7 5.3}
\arr{1.3 5.3}{1.7 5.7}
\arr{3.3 5.3}{3.7 5.7}
\arr{.3 4.3}{.7 4.7}
\arr{2.3 4.3}{2.7 4.7}
\arr{1.3 4.7}{1.7 4.3}
\arr{3.3 4.7}{3.7 4.3}
\arr{.3 3.7}{.7 3.3}
\arr{2.3 3.7}{2.7 3.3}
\arr{1.3 3.3}{1.7 3.7}
\arr{3.3 3.3}{3.7 3.7}
\arr{.3 2.3}{.7 2.7}
\arr{2.3 2.3}{2.7 2.7}
\arr{1.3 2.7}{1.7 2.3}
\arr{3.3 2.7}{3.7 2.3}
\arr{.3 1.7}{.7 1.3}
\arr{2.3 1.7}{2.7 1.3}
\arr{1.3 1.3}{1.7 1.7}
\arr{3.3 1.3}{3.7 1.7}
\arr{.3 .3}{.7 .7}
\arr{2.3 .3}{2.7 .7}
\arr{1.3 .7}{1.7 .3}
\arr{3.3 .7}{3.7 .3}
\setdots<2pt>
\plot .5 6   1.5 6 /
\plot 2.5 6  3.5 6 /
\setsolid
\plot 0 7  0 6.5 /
\plot 0 5.5  0 4.5 /
\plot 0 3.5  0 2.5 /
\plot 0 1.5  0 -1 /
\plot 4 7  4 6.5 /
\plot 4 5.5  4 4.5 /
\plot 4 3.5  4 2.5 /
\plot 4 1.5  4 -1 /
\multiput{$\vdots$} at 1 -.5  3 -.5 /
\endpicture}
$$
\end{figure}

\section{The Tiling Theorem}\label{section-tiling}

We return to the case where $M$ is a representation for a quiver
of type $\widetilde{\mathbb A}_n$. We have seen that $M$ is regular
if and only if $\per(M)>_{\mathcal F}\fin(M)$ and $\per^*(M)<_{\mathcal F}\fin^*(M)$.
We can say more.

\subsection{Three kinds of tubes}

As we have fixed an orientation of the quiver, we can define the {\it left
tube} as the exceptional tube in which irreducible monomorphisms are given
by adding a hook on the left; the {\it right tube} is similarly defined.
It is well-known that all remaining tubes are homogeneous.

\begin{proposition}\label{proposition-periodic-tubes}
An indecomposable regular module $M$ is in the left tube, in the right tube,
or in a homogeneous tube if and only if the periodic part $\per(M)$ is
$\mathcal L$, $\mathcal R$, or $(h)$, respectively, where $h=n+1$.
\end{proposition}

\begin{proof}
The result will follow from our investigation of the tubes in Section~\ref{section-tubes}.
\end{proof}

Thus, it can be read off from the pair $(\mu(M),\mu^*(M))$ whether $M$ is a regular
module, and if so in which tube $M$ occurs.  

\smallskip
If the tube is exceptional, then
the pair even determines the ray and the coray of $M$.  To visualize this, 
we introduce two partial orderings.

\bigskip
\subsection{The staircase ordering}

\smallskip
We consider families of points in the rhombic picture 
such that the limits of their measures agree and are
approached from below.  
We have seen in Chapter~\ref{section-families} that the 
families given by a ray in an exceptional tube have this 
property.

\begin{definition}
For $\mu\in\mathcal F$, let $S(\vec\mu)$ be the set of equivalence
classes of sequences $(\mu_i,\mu_i^*)$ in $\mathcal F\times\mathcal F^*$
such that the $\mu_i$ are strictly increasing with limit $\vec\mu$
and the $\mu_i^*$ are strictly decreasing. 
Two sequences are {\it equivalent} if the difference set is finite.

\smallskip
For two sequences $\mathcal M=(\mu_i,\mu_i^*),\mathcal M'=(\mu_i',\mu_i'^*)$ in $S(\vec\mu)$
we define $\mathcal M'\stairleq \mathcal M$ if for almost all $n$ there is an $\ell$
such that
$$\mu'_\ell\fleq \mu_n\flt\mu'_{\ell+1}\quad\text{and}\quad \mu_\ell'^*\fleq^*\mu_n^*.$$
\end{definition}

\begin{lemma}
For each $\mu\in\mathcal F$, the relation $\stairleq$ is a partial ordering
on $S(\mu)$.
\end{lemma}

\begin{proof}We only show anti-symmetry: Suppose $M=(\mu_i,\mu_i^*),M'=(\mu_i',\mu_i'^*)$ represent 
equivalence classes of sequences in $S(\mu)$ and satisfy $M'\stairleq M$ and $M\stairleq M'$.
We verify that $M$ and $M'$ are equivalent. 

\smallskip
There exist $L,N\in \mathbb N$ such that for each $n\geq N$ there is $\ell\geq L$ with
$$\mu_\ell'\fleq\mu_n\flt
\mu_{\ell+1}'\quad\text{and}\quad \mu_\ell'^*\fleq^*\mu_n^*$$
and such that for each $\ell\geq L$ there is $m$ with
$$\mu_m\fleq \mu'_\ell\flt\mu_{m+1}\quad\text{and}\quad  \mu_m^*\fleq^*\mu_\ell'^*.$$
Note that given $n$, both $\ell$ and $m$ are uniquely determined since the sequences $\mu_i$, $\mu_i'$
are strictly increasing.
Since $\mu_m\fleq \mu_n$, we obtain $m\leq n$ and hence $\mu_n^*\fleq^* \mu_m^*$ since the sequence
$\mu_i^*$ is strictly decreasing.  
Together with the above $\mu_m^*\fleq^*\mu_\ell'^*\fleq^* \mu_n^*$
we obtain $\mu_m^*=\mu_n^*$, hence $m=n$. It follows $\mu_n=\mu_\ell'$ and $\mu_n^*=\mu_\ell'^*$.
We have shown that $M$ and $M'$ differ at most by an initial segment.
\end{proof}

\begin{example}\label{second-example}
As second example we 
consider the path algebra $kQ$ for the quiver
$$
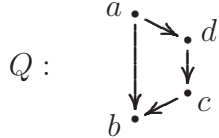
\beginpicture
  \setcoordinatesystem units <0.7cm, 0.7cm>
  \put{$Q:$} at 0 0
  \multiput{$\sssize\bullet$} at 2 1  3 .5  3 -.5  2 -1 /
  \arr{2.2 .9} {2.8 .6}
  \arr{3 .3} {3 -.3} 
  \arr{2.8 -.6}{2.2 -.9}
  \arr{2 .8}{2 -.8}
  \put{$a$} at 1.6 1.1
  \put{$b$} at 1.6 -1.1
  \put{$c$} at 3.3 -.7
  \put{$d$} at 3.4 .7
\endpicture
$$

The module $cc_1$ is simple regular; the modules on its ray,
$$cc_1,\; cd_2, \;cb_4, \;cc_5, \;cd_6, \;cb_8, \;cc_9,\;\ldots$$ 
have the following positions
in the rhombic picture.  They occur in three families, which
are indicated by solid lines in Figure~\ref{figure-stairs-regular}.
To emphasize the shape of the staircase, we rotate the picture by
$-45^\circ$.

\begin{figure}[ht]
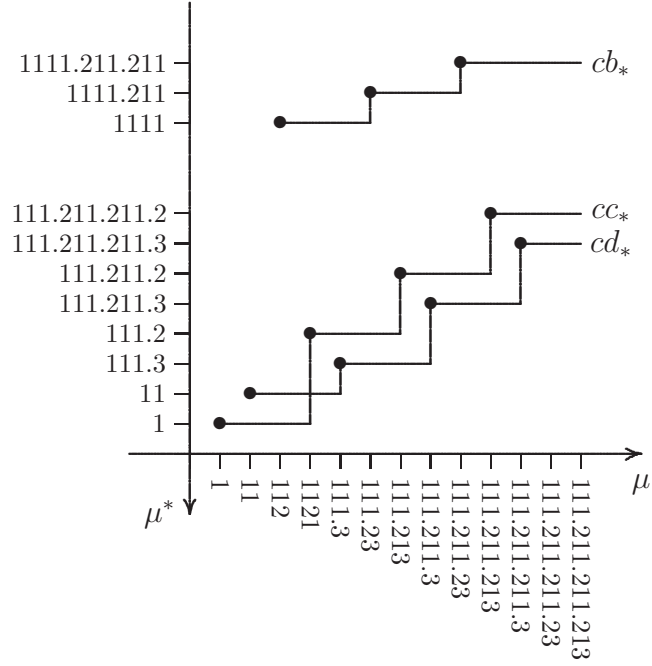

\caption{\label{figure-stairs-regular}The staircase ordering for regular modules}
$$
\hbox{\beginpicture
\setcoordinatesystem units <0.4cm,0.4cm>
\setplotarea x from 0 to 13, y from 0 to 8
\axis bottom ticks unlabeled quantity 14 /
\axis left ticks unlabeled quantity 9 /
\setplotarea x from 0 to 13, y from 11 to 13
\axis left ticks unlabeled quantity 3 /
\put{} at -5 -8
\put{} at 15 15
\arr{-2 0}{15 0}
\arr{0 15}{0 -2}
\put{$\mu^*$} at -1 -2
\put{$\mu$} at 15 -1
\put{\footnotesize\makebox[0pt][r]{1}} at -.8 1
\put{\footnotesize\makebox[0pt][r]{11}} at -.8 2
\put{\footnotesize\makebox[0pt][r]{111.3}} at -.8 3
\put{\footnotesize\makebox[0pt][r]{111.2}} at -.8 4
\put{\footnotesize\makebox[0pt][r]{111.211.3}} at -.8 5
\put{\footnotesize\makebox[0pt][r]{111.211.2}} at -.8 6
\put{\footnotesize\makebox[0pt][r]{111.211.211.3}} at -.8 7
\put{\footnotesize\makebox[0pt][r]{111.211.211.2}} at -.8 8
\put{\footnotesize\makebox[0pt][r]{1111}} at -.8 11
\put{\footnotesize\makebox[0pt][r]{1111.211}} at -.8 12
\put{\footnotesize\makebox[0pt][r]{1111.211.211}} at -.8 13
\put{\footnotesize\rotatebox{-90}{\makebox[0pt][l]{1}}} at 1 -.8
\put{\footnotesize\rotatebox{-90}{\makebox[0pt][l]{11}}} at 2 -.8
\put{\footnotesize\rotatebox{-90}{\makebox[0pt][l]{112}}} at 3 -.8
\put{\footnotesize\rotatebox{-90}{\makebox[0pt][l]{1121}}} at 4 -.8
\put{\footnotesize\rotatebox{-90}{\makebox[0pt][l]{111.3}}} at 5 -.8
\put{\footnotesize\rotatebox{-90}{\makebox[0pt][l]{111.23}}} at 6 -.8
\put{\footnotesize\rotatebox{-90}{\makebox[0pt][l]{111.213}}} at 7 -.8
\put{\footnotesize\rotatebox{-90}{\makebox[0pt][l]{111.211.3}}} at 8 -.8
\put{\footnotesize\rotatebox{-90}{\makebox[0pt][l]{111.211.23}}} at 9 -.8
\put{\footnotesize\rotatebox{-90}{\makebox[0pt][l]{111.211.213}}} at 10 -.8
\put{\footnotesize\rotatebox{-90}{\makebox[0pt][l]{111.211.211.3}}} at 11 -.8
\put{\footnotesize\rotatebox{-90}{\makebox[0pt][l]{111.211.211.23}}} at 12 -.8
\put{\footnotesize\rotatebox{-90}{\makebox[0pt][l]{111.211.211.213}}} at 13 -.8
\multiput{$\bullet$} at 1 1  2 2  3 11  4 4  5 3  6 12  7 6  8 5  9 13  10 8  11 7 /
\plot 1 1  4 1  4 4  7 4  7 6  10 6  10 8  13 8 /
\plot 2 2  5 2  5 3  8 3  8 5  11 5  11 7  13 7 /
\plot 3 11  6 11  6 12  9 12  9 13  13 13 /
\put{$cc_*$} at 14 8
\put{$cd_*$} at 14 7
\put{$cb_*$} at 14 13
\endpicture}$$
\end{figure}

The families of regular modules satisfy 
$cd_*\leq_{\rm stair} cc_*\leq_{\rm stair} cb_*$.  
Note $cc_*$ and $cd_*$ approach the same
limit in the rhombic picture.

\smallskip
By comparison, families in the preprojective component with the same 
rhombic limit may be incomparable with respect to $\leq_{\rm stair}$.
In Figure~\ref{figure-stairs-preprojective}, we picture one ray starting at the simple projective module $bb_1$.
The families $bc_*$ and $bd_*$ illustrate the situation mentioned.

\begin{figure}[ht]
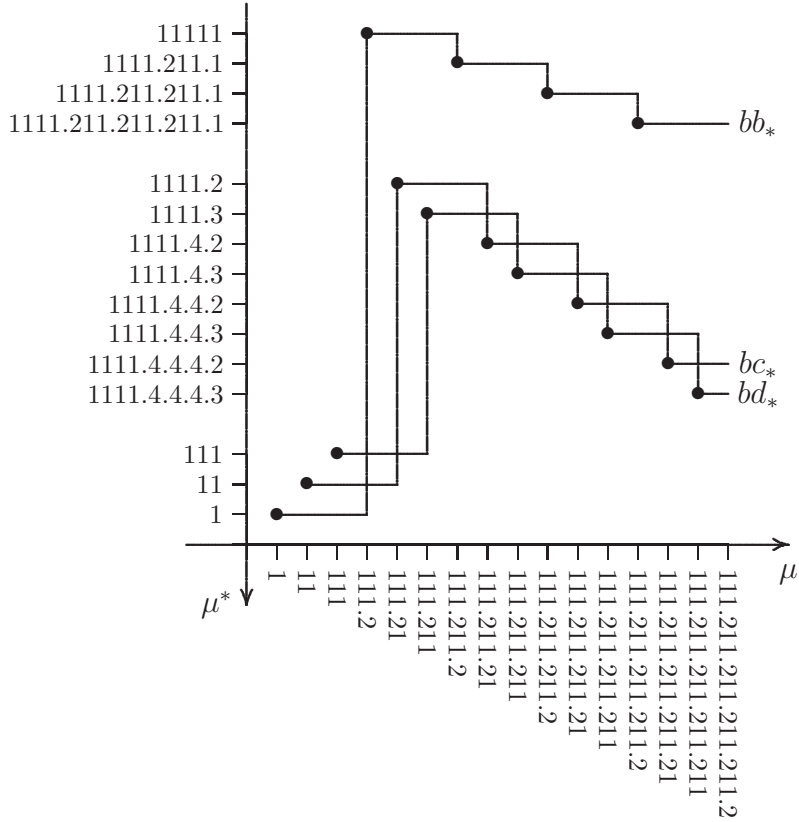

\caption{\label{figure-stairs-preprojective}The staircase ordering for preprojective modules}
$$
\hbox{\beginpicture
\setcoordinatesystem units <0.4cm,0.4cm>
\setplotarea x from 0 to 16, y from 0 to 3
\axis bottom ticks unlabeled quantity 17 /
\axis left ticks unlabeled quantity 4 /
\setplotarea x from 0 to 13, y from 5 to 12
\axis left ticks unlabeled quantity 8 /
\setplotarea x from 0 to 13, y from 14 to 17
\axis left ticks unlabeled quantity 4 /
\put{} at -5 -8
\put{} at 15 15
\arr{-2 0}{18 0}
\arr{0 18}{0 -2}
\put{$\mu^*$} at -1 -2
\put{$\mu$} at 18 -1
\put{\footnotesize\makebox[0pt][r]{1}} at -.8 1
\put{\footnotesize\makebox[0pt][r]{11}} at -.8 2
\put{\footnotesize\makebox[0pt][r]{111}} at -.8 3
\put{\footnotesize\makebox[0pt][r]{1111.4.4.4.3}} at -.8 5
\put{\footnotesize\makebox[0pt][r]{1111.4.4.4.2}} at -.8 6
\put{\footnotesize\makebox[0pt][r]{1111.4.4.3}} at -.8 7
\put{\footnotesize\makebox[0pt][r]{1111.4.4.2}} at -.8 8
\put{\footnotesize\makebox[0pt][r]{1111.4.3}} at -.8 9
\put{\footnotesize\makebox[0pt][r]{1111.4.2}} at -.8 10
\put{\footnotesize\makebox[0pt][r]{1111.3}} at -.8 11
\put{\footnotesize\makebox[0pt][r]{1111.2}} at -.8 12
\put{\footnotesize\makebox[0pt][r]{1111.211.211.211.1}} at -.8 14
\put{\footnotesize\makebox[0pt][r]{1111.211.211.1}} at -.8 15
\put{\footnotesize\makebox[0pt][r]{1111.211.1}} at -.8 16
\put{\footnotesize\makebox[0pt][r]{11111}} at -.8 17
\put{\footnotesize\rotatebox{-90}{\makebox[0pt][l]{1}}} at 1 -.8
\put{\footnotesize\rotatebox{-90}{\makebox[0pt][l]{11}}} at 2 -.8
\put{\footnotesize\rotatebox{-90}{\makebox[0pt][l]{111}}} at 3 -.8
\put{\footnotesize\rotatebox{-90}{\makebox[0pt][l]{111.2}}} at 4 -.8
\put{\footnotesize\rotatebox{-90}{\makebox[0pt][l]{111.21}}} at 5 -.8
\put{\footnotesize\rotatebox{-90}{\makebox[0pt][l]{111.211}}} at 6 -.8
\put{\footnotesize\rotatebox{-90}{\makebox[0pt][l]{111.211.2}}} at 7 -.8
\put{\footnotesize\rotatebox{-90}{\makebox[0pt][l]{111.211.21}}} at 8 -.8
\put{\footnotesize\rotatebox{-90}{\makebox[0pt][l]{111.211.211}}} at 9 -.8
\put{\footnotesize\rotatebox{-90}{\makebox[0pt][l]{111.211.211.2}}} at 10 -.8
\put{\footnotesize\rotatebox{-90}{\makebox[0pt][l]{111.211.211.21}}} at 11 -.8
\put{\footnotesize\rotatebox{-90}{\makebox[0pt][l]{111.211.211.211}}} at 12 -.8
\put{\footnotesize\rotatebox{-90}{\makebox[0pt][l]{111.211.211.211.2}}} at 13 -.8
\put{\footnotesize\rotatebox{-90}{\makebox[0pt][l]{111.211.211.211.21}}} at 14 -.8
\put{\footnotesize\rotatebox{-90}{\makebox[0pt][l]{111.211.211.211.211}}} at 15 -.8
\put{\footnotesize\rotatebox{-90}{\makebox[0pt][l]{111.211.211.211.211.2}}} at 16 -.8
\multiput{$\bullet$} at 1 1  2 2  3 3  4 17  5 12  6 11  7 16  8 10  9 9  
   10 15  11 8  12 7  13 14  14 6  15 5 /
\plot 1 1  4 1  4 17  7 17  7 16  10 16  10 15  13 15 13 14  16 14 /
\plot 2 2  5 2  5 12  8 12  8 10  11 10  11 8  14 8  14 6  16 6  /
\plot 3 3  6 3  6 11  9 11  9 9  12 9  12 7  15 7  15 5  16 5 /
\put{$bb_*$} at 17 14
\put{$bc_*$} at 17 6
\put{$bd_*$} at 17 5
\endpicture}$$
\end{figure}
\end{example}

\subsection{The ordering given by the waist-free parts}

We have seen that for a regular module $M$, the periodic part
$\per(M)$ determines the type the tube, while the multiplicity
$\mult(M)$ specifies the position within its family.  
The remaining data form the {\it waist-free part} of the measure;
similarly, one can define the waist-free part of the comeasure.
$$\wf(M)=\init(M)\cdot\fin(M); \qquad \wf^*(M)=\init^*(M)\cdot\fin^*(M)$$
Note that $\wf(M)$ and $\wf^*(M)$ do not depend on the representative 
$M$ of a family; thus they allow us to define a partial ordering for
families:  For two families $(M_i)$, $(M_j')$ we define 
$(M_i)\leq^*_{\rm wf}(M_j')$ if $\wf^*(M_i)\fleq^*\wf^*(M_j')$
holds for modules $M_i$ and  $M_j'$. 

\begin{proposition}\label{proposition-equivalent-orderings}
Consider the set $\mathcal S$ of families given by a ray in one of the tubes.
\begin{enumerate}
\item The set $\mathcal S$ is totally ordered with respect to $\leq_{\rm stair}$.
\item The set $\mathcal S$ is totally ordered with respect to $\leq_{\rm wf}^*$.
\item The two orderings $\leq_{\rm stair}$ and $\leq^*_{\rm wf}$ are equivalent.
\end{enumerate}
\end{proposition}

\begin{proof}
The result will follow from our discussion of the tubes in Section~\ref{section-tubes}.
\end{proof}

Using the waist-free ordering we can
strengthen Theorem~\ref{theorem-parallelograms} 
in the sense that the parallelograms given by Auslander-Reiten sequences 
are always non-degenerate.

\begin{corollary} Let $0\to A\to B_1\oplus B_2\to C\to 0$  be an Auslander-Reiten sequence 
in a nonhomogeneous tube 
Then the rectangle with vertices 
$$(\wf(A),\wf^*(A)),\; (\wf(B_1),\wf^*(B_1)),\; (\wf(B_2),\wf^*(B_2)),\; (\wf(C),\wf^*(C))$$
in the rhombic picture is non-degenerate.\qed
\end{corollary}

\subsection{A discussion of the tubes}\label{section-tubes}

For each tube we now investigate how the GR-measures change along a coray.  
We then give the proofs for
Proposition~\ref{proposition-periodic-tubes} and 
Proposition~\ref{proposition-equivalent-orderings}.

\medskip
We assume first that the take-off direction is to the right, 
i.e.\ $\mathcal L<_{\mathcal F}\mathcal R$ holds, and then consider the case of a symmetric
quiver.

\medskip
First we deal with the left tube, where
irreducible monomorphisms are given by adding a hook to the left end of the string,
and irreducible epimorphisms are given by deleting a cohook on the right end.
Note that whenever a cohook is deleted on the right end of a string, then 
the sequence $(\rho_i)$ of right hooks used in the Greedy Algorithm ends with a part
that is smaller than the corresponding part in $\mathcal R$.  As a consequence,
the starting point for the Greedy Algorithm is near the right end of the string.

\smallskip
For modules $M$, $M'$ in two different families on the same coray
in the left tube we have:

\begin{enumerate}
\item[(L1)] The periodic part is $\per(M)=\mathcal L$.
\item[(L2)] The initial part $\init(M)$ is strictly above the take-off limit $\vec\mu_T$.
\item[(L3)] Hence the GR-limit satisfies $\vec\mu(M)>_{\mathcal F}\vec\mu_T$.
\item[(L4)] The initial parts $\init(M)$, $\init(M')$ have different lengths,
  the final parts are equal: $\fin(M)=\fin(M')$.
\item[(L5)] The inequalities are equivalent:
  $$\init(M)\flt\init(M')  \iff  \wf(M)\flt\wf(M')  \iff \vec\mu(M)\flt\vec\mu(M')$$

\end{enumerate}

\bigskip
Next we deal with the {\it right tube,} so along a ray, hooks are added on the right end
of the string, and along a coray, cohooks are deleted from the left end of the string.
From the discussion of the Greedy Algorithm it follows that the starting point is
near the left end of the string.

\smallskip
For a module $M$ in the right tube we have:

\begin{enumerate}
\item[(R1)] The periodic part is $\per(M)=\mathcal R$.
\item[(R2)] The initial part $\init(M)$ is not necessarily above the take-off limit $\vec\mu_T$.
\end{enumerate}

We consider the two cases given by (R2) separately:

\smallskip
Suppose $M$, $M'$ occur in different families on the same coray,
and have both initial parts above the take-off limit $\vec\mu_T$.
Then the final parts $\fin(M)$, $\fin(M')$ are equal and 
consist only of parts of the right hook sequence $\mathcal R$.

\begin{enumerate}
\item[(R3$^{\prime}$)] The GR-limit is above the take-off limit, $\vec\mu_T\flt\vec\mu(M)$.
\item[(R4$^{\prime}$)] The initial parts $\init(M)$, $\init(M')$ have different lengths,
  and the final parts are equal, $\fin(M)=\fin(M')$.
\item[(R5$^{\prime}$)] The inequalities are equivalent:
  $$\init(M)\flt\init(M') \iff  \wf(M)\flt\wf(M')  \iff \vec\mu(M)\flt\vec\mu(M')$$
\end{enumerate}

Now assume that $M$, $M'$ occur in different families 
on the same coray in the right tube
such that both have initial parts below the take-off limit 
$\vec\mu_T$.  Then $\fin(M)$ and $\fin(M')$ are obtained by the Greedy Algorithm
from the same set of parts from $\mathcal R$, possibly from some parts 
from $\mathcal L$, and from one short left hook.

\begin{enumerate}
\item[(R3$^{\prime\prime}$)] The GR-limit is the take-off limit, $\vec\mu_T=\vec\mu(M)$.
\item[(R4$^{\prime\prime}$)] The initial parts are equal, $\init(M)=\init(M')$;
  the final parts $\fin(M)$, $\fin(M')$ have different lengths;
  their length difference is less than $h$.
\item[(R5$^{\prime\prime}$)] The inequalities are equivalent: 
  $$\len\fin(M)>\len\fin(M') \iff  \fin(M)\flt\fin(M') \iff \wf(M)\flt\wf(M')$$
\end{enumerate}

\bigskip
Let $M$ be a quasi-simple homogeneous module (see \cite{Ringel84}), and let  $M'$ be another module 
in the family of $M$, which is just the tube of $M$.

\begin{enumerate}
\item[(H1)] The periodic part is $\per(M')=\mathcal (h)$.
\item[(H2)] The initial part $\init(M')$ equals $\mu(M)$.
\item[(H3)] The GR-limit satisfies $\vec\mu(M')>_{\mathcal F}\vec\mu_T$.
\end{enumerate}

\bigskip
It remains to deal with the case where $Q$ is a symmetric quiver, so there is a
non-trivial symmetry operation on $Q$, which gives rise to a self-equivalence
of $\mod KQ$ that permutes the two tubes.  We obtain:

\begin{enumerate}
\item[(S0)] The set of points in the rhombic picture corresponding to the modules in the left
  tube is identical with the set of points for the right tube.
\item[(S1--5)] For both tubes, the statements (L1)--(L5) hold.
\end{enumerate}

We can now complete the remaining proofs.

\begin{proof}[Proof of Proposition~\ref{proposition-periodic-tubes}]
The periodic parts of the left tube, the right tube, and a homogeneous tube are $\mathcal L$,
$\mathcal R$, and $(h)$, according to (L1), (R1), and (H1).  As indicated before (S1),
in the symmetric case the GR-data cannot distinguish the two exceptional tubes; 
for each the periodic part is $\mathcal L=\mathcal R$.
\end{proof}

\begin{proof}[Proof of Proposition~\ref{proposition-equivalent-orderings}]
Fix a ray in one of the exceptional tubes, and 
let $(M_i)$ and $(M_i')$ be two families on the ray.  Put
$\mu_i=\mu(M_i)$, $\mu_i'=\mu(M_i')$, $\mu_i^*=\mu^*(M_i)$, $\mu_i'^*=\mu^*(M_i')$.
As the measures along the ray increase monotonically, say towards $\vec\mu$, 
the families $(M_i)$, $(M_i')$ have the same GR-limit $\vec\mu=\lim_i\mu_i=\lim_i\mu'_i$.
Since the modules are regular, the periodic part is less than the final part,
so within each family the GR-comeasures are strictly decreasing.
Hence the families given by the ray can be compared
in $\mathcal S(\vec\mu)$ with respect to the orderings $\stairleq$ and $\wfleq$.

\smallskip
(1) For the first part of the Proposition, 
we have to show that for the sequences $\mathcal M=(\mu_i,\mu_i^*)$, $\mathcal M'=(\mu_i',\mu_i'^*)$
either $\mathcal M\stairleq\mathcal M'$ or $\mathcal M'\stairleq\mathcal M$ holds.
Note that $\mu_\ell'\leq\mu_n<\mu_{\ell+1}'$ if and only if $\dim M_\ell'\leq\dim M_n<\dim M_\ell'+h$ 
if and only if $\dim DM_\ell'\leq\dim DM_n<\dim DM_\ell'+h$.
The statement is clear if $\lim_n\mu_n^*\neq\lim_\ell\mu_\ell'^*$.
Suppose $\ell$, $n$ satisfy the condition $\mu_\ell'\leq\mu_n<\mu_{\ell+1}'$.  
Since the modules $M=DM_n$, $M'=DM_\ell'$ lie on the same coray, we may use the above results.
In the situations (L4) and (R4$^{\prime}$), the GR-comeasures $\mu_n^*$ and $\mu_\ell'^*$ differ in their
initial parts, and hence the limits are different.  The only possibility for above limits to be
equal is the situation in (R4$^{\prime\prime}$) where the initial and periodic parts are equal, and the final
parts differ in their lengths.  We distinguish two cases:

\smallskip
First assume $\dim M'\leq\dim M<\dim M'+h$ and $\ell\fin(M)>\ell\fin(M')$. Then by (R5$^{\prime\prime}$),
$\fin(M)\flt\fin(M')$, so $\mu_n^*\fgt^*\mu_\ell'^*$.  Thus, by definition, 
$\mathcal M'\stairleq \mathcal M$.  

\smallskip
Now assume $\dim M'\leq\dim M<\dim M'+h$ and $\ell\fin(M')>\ell\fin(M)$.
Let $M'^+=DM_{\ell+1}'$ be the next module in the family, so $\dim M\leq\dim M'^+<\dim M+h$ holds.
By (R4$^{\prime\prime}$), the length of the final parts 
of $M$ and $M'^+$ differs by less than $h$.
It follows that there is an $m\in\mathbb N$
such that 
\begin{eqnarray*}\mu(M)&=&\init(M)\cdot \per(M)^{m+1}\cdot\fin(M)\\
  \mu(M')&=&\init(M)\cdot\per(M)^m\cdot\fin(M')\\
  \mu(M'^+)&=&\init(M)\cdot\per(M)^{m+1}\cdot\fin(M')\end{eqnarray*}
hold.  Again by (R5$^{\prime\prime}$),
$\fin(M')\flt\fin(M)$, so $\mu_{\ell+1}'^*\fgt^*\mu_n^*$.  Thus, by definition, 
$\mathcal M'\stairgeq \mathcal M$.  

\medskip
(2) Clearly, different families on a coray have different waist-free parts.  
As a consequence, $\wfleq$ is a total ordering for the families on a ray in an exceptional tube.

\medskip
(3) The staircase ordering is as follows:  $\mathcal M'\stairleq \mathcal M$ if
$\lim_i\mu_i'^*\fleq^*\lim_i\mu_i^*$, and in case the limits are equal, if
$\len\fin(DM_j')\leq\len\fin(DM_i)$ holds for $i$, $j$ suitably large.  
The condition on the limits is equivalent to 
$\init(DM_j')\fgeq\init(DM_i)$, the condition on the lengths is equivalent to
$\fin(DM_j')\fgeq\fin(DM_i)$  by (R5$^{\prime\prime}$).  In conclusion, $\mathcal M'\stairleq \mathcal M$
holds if and only if  $\wf(DM_j')\fgeq\wf(DM_i)$, which is equivalent to 
$\mathcal M'\wfleq \mathcal M$.

\end{proof}

\subsection{The Tiling Theorem}

\begin{definition}
We say that a quiver $Q$ of type $\widetilde{\mathbb A}_n$ 
{\em has a unique widest valley} if in each exceptional tube,
the same sink, up to the periodic shift,
may be chosen as Greedy Algorithm starting point 
for all modules $M$ of sufficiently large dimension.

\smallskip
Dually, $Q$ {\em has a unique widest hill} if in each exceptional tube
the same sink in $Q^\op$, up the shift,
may be chosen as Greedy Algorithm starting point for each module $DM$
where $M$ is a sufficiently large module in the tube.
\end{definition}

\begin{remark}
The conditions in the definition are satisfied if $Q$ has a valley and a hill
that are at least 2 units wider than any other valley and hill, respectively.
If $Q$ has several valleys of maximal size, the starting point will jump 
among their sinks.
If the difference in size is just 1 unit, a jump among different sinks may occur. 
For the convenience of the reader, we repeat the two-sink two-source example given earlier.
\end{remark}

\begin{example}
We see that the two-sink, two-source 
$\widetilde{\mathbb A}_{4}$ quiver from Example~\ref{first-example}
has both a unique widest valley and a unique widest hill. 

$$
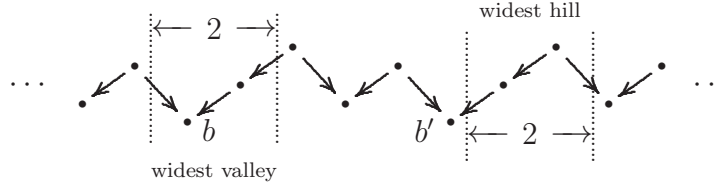
\beginpicture
  \setcoordinatesystem units <.7cm, .5cm>
  \multiput{$\sssize\bullet$} at  6 -.5  7 .5  8 -1  
                       9 0  10 1  11 -.5  12 .5  13 -1
                       14 0  15 1  16 -.5  17 .5 /
  \arr{6.8 .3} {6.2 -.3}
  \arr{7.2 .2} {7.8 -.7}
  \arr{8.8 -.2} {8.2 -.8}
  \arr{9.8 .8} {9.2 .2}
  \arr{10.2 .7} {10.8 -.2}
  \arr{11.8 .3} {11.2 -.3}
  \arr{12.2 .2} {12.8 -.7}
  \arr{13.8 -.2} {13.2 -.8}

  \arr{14.8 .8} {14.2 .2}
  \arr{15.2 .7} {15.8 -.2}
  \arr{16.8 .3} {16.2 -.3}

  \put{$b$} at 8.4 -1.2
  \put{$b'$} at 12.5 -1.2
\setdots<2pt>
\plot 7.3 -1.5  7.3 2 /
\plot 9.7 -1.5  9.7 2 /
\plot 13.3 -2  13.3 1.5 /
\plot 15.7 -2  15.7 1.5 /
\betweenarrows {2} from 7.3 1.5  to 9.7 1.5 
\betweenarrows {2} from 13.3 -1.3  to 15.7 -1.3
\put{\tiny widest valley} at 8.5 -2.3
\put{\tiny widest hill} at 14.5 2 
\multiput{$\cdots$} at 5 0  18 0 /
\endpicture
$$
\end{example}

With this notion, we can add the following observation to our discussion of the tubes.

\smallskip
\begin{remark}
Suppose $Q$ has a unique widest valley and assume that the modules $M$, $M'$ occur
on the same coray in an exceptional tube for $kQ$.
Then $\init(M)\flt\init(M')$ if and only if $\ell\init(M)>\ell\init(M')$.
Here we denote by $\ell$ the sum of the parts of a finite sequence in $\mathcal F$.
\end{remark}

\medskip
We now restate our final result, The Tiling Theorem.\\

\begin{theorem}\label{theorem-tiling}
Suppose the quiver $Q$ of type $\widetilde{\mathbb A}_n$ has a unique
widest valley and a unique widest hill. Then for each tube,
the system of limits in the rhombic picture provides a tiling for the tube.
\end{theorem}

A technical lemma is needed before we can commence the proof of the theorem.

\begin{lemma}
Suppose all sufficiently large modules on a coray in a non-homogeneous tube 
have the same Greedy Algorithm starting point, up to the periodic shift.
Then the ordering of the waist-free parts of the families agrees with 
the sequence in which they occur along the coray.
\end{lemma}

\begin{proof}
Along a coray, the lengths of the waist-free parts follow a saw-tooth pattern as they
decrease until the starting point ``jumps''. We have seen in the discussion of the tubes
that as the lengths of the waist-free parts decrease, their measures 
increase in $\mathcal F$.
Thus, by choosing a suitable first family on the coray,
the ordering of the families agrees with the ordering of their the waist-free parts.

\end{proof}

\begin{proof}[Proof of the Tiling Theorem~\ref{theorem-tiling}]
We have already seen that the ordering of families by their waist-free parts agrees with the
staircase ordering which we picture as a
left-right ordering in the rhombic picture.
Using the lemma and its dual version we obtain the result.
\end{proof}

\subsection{Examples}

\begin{example}
The quiver $Q$ considered in Example~\ref{first-example}
is symmetric with respect to rotation by $\pi$, so the take-off direction
of $Q$, and the take-off direction of $Q^\op$ are both to the right.  

\smallskip
The left tube pictured in Figure~\ref{figure-left}
consists of the regular modules with with periodic parts $\per=32$ and $\per^*=32$.
The  GR-limits in the 
rhombic picture in Figure~\ref{figure-rhombic-one} 
are above the take-off limit and their GR-colimits are below the take-off colimit,
hence the limit points are determined uniquely by the inital parts (see (L2), (L4) above).

\smallskip
The right tube is given by $\per=221$, $\per^*=221$. Here the take-off limit and the take-off
colimits are attained, but only by one family on each coray and each ray, respectively.
So the limit points are still determined uniquely by the inital parts (see (R2), (R4$^{\prime}$), 
(R3$^{\prime\prime}$), (R4$^{\prime\prime}$) above).
The picture in the introduction shows how the
system of limit points in the rhombic picture corresponds to a tiling of the tube.
\end{example}

\begin{example}
We now discuss briefly the rhombic picture in Figure~\ref{figure-second-rhombic-picture}
for the one sink-one source quiver
in Example~\ref{second-example}.
Again, the take-off direction and the take-off codirection are both to the right.  
Here, the left tube consists only of one family, $ba_*$.

\smallskip
The right tube, which consists of the nine 
families $ab_*$, $ac_*$, $ad_*$, $cb_*$, $cc_*$, $cd_*$, $db_*$, $dc_*$, $dd_*$,
has the property that on each ray, the comeasures for two of three 
families approach the take-off colimit, and on each coray, the measures for two families
approach the take-off limit.  Thus in the rhombic picture, there are only four limit points
for nine families.

\smallskip
Along the ray $cc_1\to cd_2\to cb_4\to cc_5\to \cdots$, the families $cc_*$, $cd_*$ have the same
colimits, but we can use the staircase ordering to refine the colimit by
taking also the direction from where the colimit
is approached into account, as we have seen in Example~\ref{second-example}.
Using this refinement, the system of limits in the rhombic picture produces a tiling of the 
exceptional tube pictured in Figure~\ref{figure-exceptional}.
\end{example}

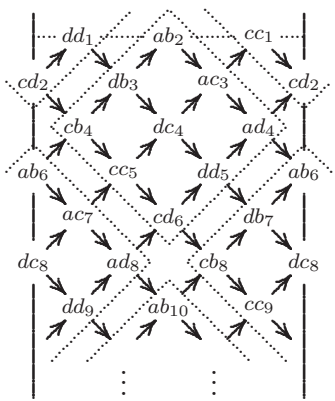
\begin{figure}[ht]
\caption{\label{figure-second-rhombic-picture}The rhombic picture for the second example}
$$
\hbox{\beginpicture
\setcoordinatesystem units <0.3cm,0.3cm>
\multiput{} at -16 16  17 16  3 -2  1 29 /
\arr{-16 16}{2 -2}
\put{$\mu^*$} at 3 -2
\arr{-2 -2}{16 16}
\put{$\mu$} at 17 16

\plot -6.3 5.7  -6 6 /
\put{$\vec\mu_T^*$} at -7 5
\setdots<3pt>
\plot -1 11  1 13 / \plot 2 14  9 21 /
\setsolid

\plot -9.3 8.7  -9 9  /
\put{$\vec\mu_H^*$} at -10 8
\setdots<3pt>
\plot -4 14  1 19 /
\setsolid

\plot -14.3 13.7 -14 14 /
\put{$\vec\mu_L^*$} at -15 13 
\setdots<3pt>
\plot -9 19  -7 21 /  \plot -1 27  1 29 /
\setsolid

\plot 6.3 5.7  6 6 /
\put{$\vec\mu_T$} at 7 5
\setdots<3pt>
\plot 1 11  -1 13  / \plot -2 14  -9 21 /
\setsolid

\plot 9.3 8.7  9 9  /
\put{$\vec\mu_H$} at 10 8
\setdots<3pt>
\plot 4 14  -1 19 /
\setsolid

\plot 14.3 13.7 14 14 /
\put{$\vec\mu_L$} at 15 13 
\setdots<3pt>
\plot 9 19  7 21 /  \plot 1 27  -1 29 /
\setsolid

\multiput{\makebox(0,0){$\bullet$}} at 0 12  3 15  8 20  -3 15  0 18  -8 20  0 28 /

\put{\makebox(0,0){$\begin{blue}ab\end{blue}\strut$}} at 0 24
\arr{0 25}{0 27}
\put{\makebox(0,0){$\begin{blue}db\end{blue}\strut$}} at -8 16
\arr{-8 17}{-8 19}
\put{\makebox(0,0){$\begin{blue}cb\end{blue}\strut$}} at -9 15
\arr{-8.75 16.2}{-8.25 19}
\put{\makebox(0,0){$\begin{blue}ac\end{blue}\strut$}} at 8 16
\arr{8 17}{8 19}
\put{\makebox(0,0){$\begin{blue}ad\end{blue}\strut$}} at 9 15
\arr{8.75 16.2}{8.25 19}
\put{\makebox(0,0){$\begin{blue}dc\end{blue}\strut$}} at 0 8
\arr{0 9}{0 11}
\put{\makebox(0,0){$\begin{blue}cc\end{blue}\strut$}} at -1 7
\arr{-.8 8}{-.2 11}
\put{\makebox(0,0){$\begin{blue}dd\end{blue}\strut$}} at 1 7
\arr{.8 8}{.2 11}
\put{\makebox(0,0){$\begin{blue}cd\end{blue}\strut$}} at 0 6

\arr{-12 20}{-9 20}
\put{\makebox(0,0){$bb\strut$}} at -11 19
\arr{-7 15}{-4 15}
\put{\makebox(0,0){$bc,bd\strut$}} at -6 14
\arr{12 20}{9 20}
\put{\makebox(0,0){$aa\strut$}} at 11 19
\arr{7 15}{4 15}
\put{\makebox(0,0){$ca,da\strut$}} at 6 14
\arr{0 15}{0 17}
\put{\makebox(0,0){$H, ba\strut$}} at 0 14
\endpicture}$$
\end{figure}

\begin{figure}[ht]
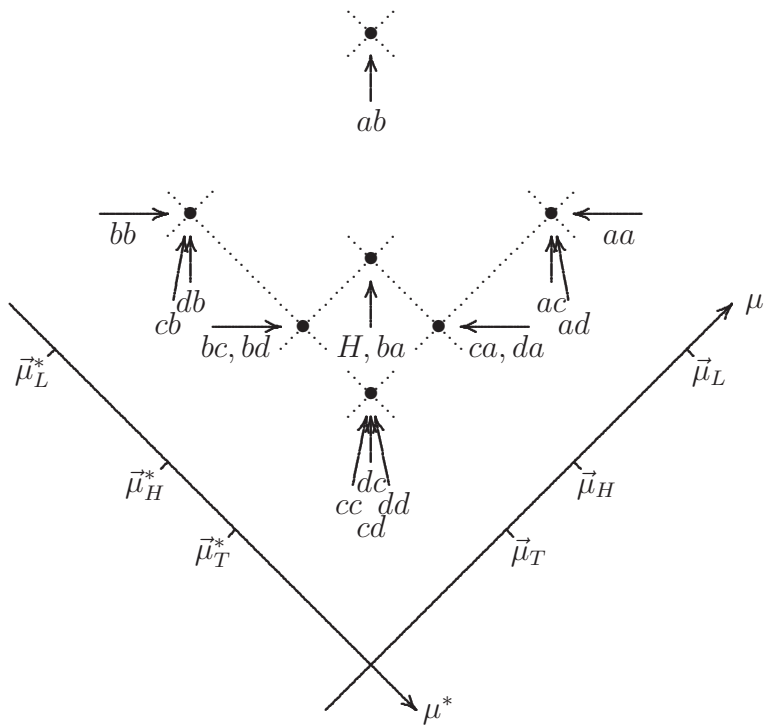

\caption{\label{figure-exceptional}The right tube in the second example}
$$\hbox{\beginpicture
\setcoordinatesystem units <0.6cm,0.6cm>
\put{} at -2 8.6
\put{} at 6 0
\put{$\ssize dd_1$} at 0 8
\put{$\ssize \begin{blue}ab_2\end{blue}$} at 2 8
\put{$\ssize cc_1$} at 4 8
\put{$\ssize cd_2$} at -1 7
\put{$\ssize \begin{blue}db_3\end{blue}$} at 1 7
\put{$\ssize \begin{blue}ac_3\end{blue}$} at 3 7
\put{$\ssize cd_2$} at 5 7
\put{$\ssize \begin{blue}cb_4\end{blue}$} at 0 6
\put{$\ssize \begin{blue}dc_4\end{blue}$} at 2 6
\put{$\ssize \begin{blue}ad_4\end{blue}$} at 4 6
\put{$\ssize ab_6$} at -1 5
\put{$\ssize \begin{blue}cc_5\end{blue}$} at 1 5
\put{$\ssize \begin{blue}dd_5\end{blue}$} at 3 5
\put{$\ssize ab_6$} at 5 5
\put{$\ssize ac_7$} at 0 4
\put{$\ssize \begin{blue}cd_6\end{blue}$} at 2 4
\put{$\ssize db_7$} at 4 4
\put{$\ssize dc_8$} at -1 3
\put{$\ssize ad_8$} at 1 3
\put{$\ssize cb_8$} at 3 3
\put{$\ssize dc_8$} at 5 3
\put{$\ssize dd_9$} at 0 2
\put{$\ssize ab_{10}$} at 2 2
\put{$\ssize cc_9$} at 4 2
\arr{.3 7.7}{.7 7.3}
\arr{2.3 7.7}{2.7 7.3}
\arr{4.3 7.7}{4.7 7.3}
\arr{1.3 7.3}{1.7 7.7}
\arr{3.3 7.3}{3.7 7.7}
\arr{-.7 7.3}{-.3 7.7}
\arr{.3 6.3}{.7 6.7}
\arr{2.3 6.3}{2.7 6.7}
\arr{4.3 6.3}{4.7 6.7}
\arr{1.3 6.7}{1.7 6.3}
\arr{3.3 6.7}{3.7 6.3}
\arr{-.7 6.7}{-.3 6.3}
\arr{.3 5.7}{.7 5.3}
\arr{2.3 5.7}{2.7 5.3}
\arr{4.3 5.7}{4.7 5.3}
\arr{1.3 5.3}{1.7 5.7}
\arr{3.3 5.3}{3.7 5.7}
\arr{-.7 5.3}{-.3 5.7}
\arr{.3 4.3}{.7 4.7}
\arr{2.3 4.3}{2.7 4.7}
\arr{4.3 4.3}{4.7 4.7}
\arr{1.3 4.7}{1.7 4.3}
\arr{3.3 4.7}{3.7 4.3}
\arr{-.7 4.7}{-.3 4.3}
\arr{.3 3.7}{.7 3.3}
\arr{2.3 3.7}{2.7 3.3}
\arr{4.3 3.7}{4.7 3.3}
\arr{1.3 3.3}{1.7 3.7}
\arr{3.3 3.3}{3.7 3.7}
\arr{-.7 3.3}{-.3 3.7}
\arr{.3 2.3}{.7 2.7}
\arr{2.3 2.3}{2.7 2.7}
\arr{4.3 2.3}{4.7 2.7}
\arr{1.3 2.7}{1.7 2.3}
\arr{3.3 2.7}{3.7 2.3}
\arr{-.7 2.7}{-.3 2.3}
\arr{.3 1.7}{.7 1.3}
\arr{2.3 1.7}{2.7 1.3}
\arr{4.3 1.7}{4.7 1.3}
\arr{1.3 1.3}{1.7 1.7}
\arr{3.3 1.3}{3.7 1.7}
\arr{-.7 1.3}{-.3 1.7}
\setdots<2pt>
\plot -1 8  -.5 8 /
\plot .5 8  1.5 8 /
\plot 2.5 8  3.5 8 /
\plot 4.5 8  5 8 /
\setsolid
\plot -1 8.5  -1 7.5 /
\plot -1 6.5  -1 5.5 /
\plot -1 4.5  -1 3.5 /
\plot -1 2.5  -1 0 /
\plot 5 8.5  5 7.5 /
\plot 5 6.5  5 5.5 /
\plot 5 4.5  5 3.5 /
\plot 5 2.5  5 0 /
\multiput{$\vdots$} at 1 .5  3 .5 /
\setdots<2pt>
  \plot -.6 6  2 8.6  4.6 6  2 3.4  -.6 6 /
  \plot -1.6 7  -1 6.4  1.2 8.6 /
  \plot 2.8 8.6  5 6.4  5.6 7 /
  \plot 5.6 5  5 5.6  2.4 3  4.6 .8 /
  \plot -1.6 5  -1 5.6  1.6 3  -.6 .8 /
  \plot .2 .8  2 2.6  3.8 .8 /
\endpicture}
$$
\end{figure}

\newpage
\section{Acknowledgements}
The first named author wishes to thank Manhattan College for the invitation to
two monthlong visits in the summers of 2010 and 2011. 
The project was initiated when the second named author visited Florida Atlantic University during her sabbatical leave from Manhattan College in the spring of 2010. She is grateful to both institutions for their generosity and hospitality. Both authors appreciate the helpful suggestions offered by the referee.

\end{document}